\documentclass{article}

\usepackage{times}

\usepackage{soul}
\usepackage{url}
\usepackage[hidelinks]{hyperref}
\usepackage[utf8]{inputenc}
\usepackage[small]{caption}
\usepackage{subcaption}
\usepackage{graphicx}
\usepackage{amsmath}
\usepackage{comment}
\usepackage{booktabs}
\usepackage{graphicx}
\usepackage{algorithm}
\usepackage[noend]{algorithmic}
\usepackage{tikz}
\usepackage{diagbox}
\usepackage{amsmath,amsfonts,amssymb,times,amsthm,bm}
\newtheorem{theorem}{Theorem}[section]
\newtheorem{lemma}[theorem]{Lemma}

\newtheorem{problem}[theorem]{Problem}
\newtheorem{conjecture}[theorem]{Conjecture}
\newtheorem{proposition}[theorem]{Proposition}
\theoremstyle{definition}
\usetikzlibrary{calc,arrows.meta,positioning,shapes.geometric}

\date{}

\title{Infinite families of planar graphs of a given injective chromatic number}

\author{\sc
Matias DANEELS\footnote{Department of Computer Science, KU Leuven Campus Kulak-Kortrijk, 8500 Kortrijk, Belgium}\;,
Jan GOEDGEBEUR\footnotemark[1]\;\footnote{Department of Applied Mathematics, Computer Science and Statistics, Ghent University, 9000 Ghent, Belgium}\;
and Jarne RENDERS\footnotemark[1]\;
\footnote{E-mail addresses:
    matias.daneels@student.kuleuven.be,
    \{jan.goedgebeur,
    jarne.renders\}@kuleuven.be
    }
}

\begin{document}

\maketitle

\begin{center}
\begin{minipage}{125mm}
\textbf{Abstract.}
An \emph{injective colouring} of a graph is a colouring in which every two vertices sharing a common neighbour receive a different colour. Chen, Hahn, Raspaud and Wang conjectured that every planar graph of maximum degree $\Delta \ge 3$ admits an injective colouring with at most $\lceil 3\Delta/2\rceil$ colours. This was later disproved by Lu\v zar and \v Skrekovski for certain small and even values of $\Delta$ and they proposed a new refined conjecture.
Using an algorithm for determining the \emph{injective chromatic number} of a graph, i.e.\ the smallest number of colours for which the graph admits an injective colouring, we give computational evidence for Lu\v zar and \v Skrekovski's conjecture and extend their results by presenting an infinite family of $3$-connected planar graphs for each $\Delta$ (except for $4$) attaining their bound,
whereas they only gave finitely many examples for each $\Delta$. Hence, together with another infinite family of maximum degree $4$, we provide infinitely many counterexamples to the conjecture by Chen et al.\ for each $\Delta$ if $4\le \Delta \le 7$ and every even $\Delta \ge 8$.
We provide similar evidence for analogous conjectures by La and \v Storgel and Lu\v zar, \v Skrekovski, and Tancer when the girth is restricted as well. Also in these cases we provide infinite families of $3$-connected planar graphs attaining the bounds of these conjectures for certain maximum degrees $\Delta\geq 3$.

\bigskip

\bigskip


\end{minipage}
\end{center}


\section{Introduction}\label{sect:introduction}

In this paper all graphs will be simple and undirected.
For $k > 0$, an \emph{injective $k$-colouring} of a graph $G$ is a function $c:V(G)\rightarrow \{1,\ldots, k\}$ such that for all pairs of distinct vertices $v,w\in V(G)$ sharing a common neighbour, $c(v)\neq c(w)$.
The \textit{injective chromatic number} of a graph $G$, denoted by $\chi_i(G)$, is the smallest $k$ for which $G$ has an injective $k$-colouring. These notions were first introduced by Hahn, Kratochv\'il, \v Sir\'a\v n, and Sotteau~\cite{HKSS02} in 2002. In 2012, Chen, Hahn, Raspaud, and Wang~\cite{chen2012some} investigated $K_4$-minor-free graphs, a subclass of planar graphs, and showed that if $G$ is such a graph of maximum degree $\Delta$, then $\chi_i(G) \leq \left\lceil \frac{3}{2} \Delta \right\rceil$. They conjectured that the same bound holds for all planar graphs.

\begin{conjecture}[Chen, Hahn, Raspaud, and Wang~\cite{chen2012some}]
For every planar graph $G$ with maximum degree $\Delta$ it holds that
\[\chi_i(G) \leq \left\lceil \dfrac{3}{2} \Delta \right\rceil.\]
\label{conj:chen}
\end{conjecture}

However, in 2015, Lu\v{z}ar and \v{S}krekovski~\cite{luvzar2015counterexamples} disproved Conjecture~\ref{conj:chen} by providing a (finite) number of counterexamples for each $\Delta$ with $4\le \Delta \le 7$ and each even $\Delta\ge 8$, leaving the original conjecture open in particular for $\Delta = 3$.
Using their counterexamples as an upper bound, Lu\v{z}ar and \v{S}krekovski proposed a new conjecture, similar to the famous conjecture by Wegner~\cite{wegner1977graphs} for the related \emph{$2$-distance colourings}, where no two vertices at distance at most $2$ can receive the same colour.

\begin{conjecture}[Lu\v{z}ar and \v{S}krekovski~\cite{luvzar2015counterexamples}]
 For every planar graph $G$ with maximum degree $\Delta$ it holds that
\[
  \chi_i(G) \leq
     \begin{cases}
       5,  & \text{if } \Delta \leq 3;\\
       \Delta + 5,  & \text{if } 4 \leq \Delta \leq 7;\\
       \lfloor \frac{3}{2} \Delta \rfloor + 1, & \text{if } \Delta \geq 8.\\
     \end{cases}
\]
\label{conj:luzar}
\end{conjecture}

When restricting oneself to graphs of higher girths, Conjecture~\ref{conj:chen} may still be true and even stronger conjectures have been posed in this case. (The \textit{girth} of a graph is the length of its shortest cycle.) La and \v Storgel~\cite{la20232} conjectured that for planar graphs $G$ with girth at least $4$ and maximum degree $\Delta\ge 3$
it holds that
$\chi_i(G)\le \lfloor\frac{3}{2}\Delta\rfloor$. For girth at least $5$, Lu\v{z}ar, \v{S}krekovski, and Tancer~\cite{LST09} asked if there is an integer $M$ such that every planar graph $G$ of girth at least $5$ with maximum degree $\Delta\ge M$ is injectively $(\Delta + 1)$-colourable. This mirrors a similar conjecture by Molloy and Salavatipour~\cite{MS05} in the $2$-distance colouring case, where certain upper bounds can be shown if the maximum degree is sufficiently large.

The counterexamples to Conjecture~\ref{conj:chen} which were found by Lu\v{z}ar and \v{S}krekovski~\cite{luvzar2015counterexamples} are all of \emph{diameter} $2$, i.e.\ the maximum length of a shortest path between two vertices is exactly $2$, and they use this fact to show that the injective chromatic number of these graphs is equal to their order. One can easily see that for each maximum degree $\Delta$ only a finite number of diameter $2$ counterexamples can exist.
Hence, this begs the question of whether the conjecture can be saved by excluding graphs of diameter~$2$ or by requiring some other properties and of whether we can find infinitely many counterexamples to Conjecture~\ref{conj:chen} for each maximum degree.

The \emph{connectivity} of a graph is the minimum number of vertices needed to remove before the graph becomes disconnected. If a graph with more than $k$ vertices cannot be disconnected by removing fewer than $k$ vertices, we say it is \emph{$k$-connected}.
As we will see later, for any $\Delta \geq 3$ (except $\Delta = 4$) infinitely many examples of connectivity $1$ or $2$ attaining the bound of Conjecture~\ref{conj:luzar}, can easily be obtained, but we will show that requiring the connectivity to be $3$ is also not sufficient to save Conjecture~\ref{conj:chen}.


It has been shown that planar graphs $G$ of girth at least~$7$ and sufficiently high maximum degree $\Delta$ have $\chi_i(G) = \Delta$~\cite{BI11}. When the girth is at least $6$, $\chi_i(G)\leq \Delta +1$ for sufficiently high maximum degree $\Delta$~\cite{DL13} and for any $\Delta \geq 3$ there exist examples of planar graphs
of girth $6$ with $\chi_i(G) = \Delta +1$~\cite{LST09}.
When the girth is at least $5$, $\chi_i(G)\leq \Delta + 3$ for sufficiently large $\Delta$~\cite{BY22, DL14}, but to the best of our knowledge no planar girth $5$ examples with $\chi_i(G) > \Delta + 1$ are known.
For a more detailed overview, see~\cite[Table~1]{la20232}.

In this paper we study Conjecture~\ref{conj:luzar} and similar questions for girth at least $4$ and $5$ using computational means and verify them for small orders.
We do this by designing and implementing an exact backtracking algorithm, which can be found on GitHub~\cite{DGR24Program}, to calculate the injective chromatic number of a graph and use it to determine the injective chromatic number of all small planar graphs. For each $\Delta \ge 3$, except for $4$, we present infinite families of planar graphs in which every graph has maximum degree $\Delta$ and which attain the bound for the injective chromatic number proposed by Lu\v{z}ar and \v{S}krekovski in Conjecture~\ref{conj:luzar}.
For $\Delta = 4$, we present an infinite family of planar graphs which are counterexamples to Conjecture~\ref{conj:chen} by Chen et al., but which are one less than the bound proposed by Lu\v{z}ar and \v{S}krekovski in Conjecture~\ref{conj:luzar}. Moreover, every graph in each of our families is $3$-connected. Hence, we also give for each maximum degree $\Delta$ with $4\leq \Delta \leq 7$ and every even $\Delta \geq 8$ an infinite family of $3$-connected graphs which are counterexamples to Conjecture~\ref{conj:chen}.


The rest of this paper is organised as follows.
In Section~\ref{sect:results}, we briefly describe the algorithm we used to compute the injective chromatic number of a graph and provide computational results showing that Conjecture~\ref{conj:luzar} holds up to order at least $13$ and show that a cubic counterexample to Conjecture~\ref{conj:chen} must have at least $38$ vertices.
Furthermore, we also correct a slight error from~\cite{luvzar2015counterexamples} concerning a graph in Figure 3 of their paper.

In Section~\ref{sec:families}, we provide infinitely many counterexamples to Conjecture~\ref{conj:chen} when $\Delta = 4$ and infinite families of graphs
which are sharp for Conjecture~\ref{conj:luzar} in the case of $\Delta = 3$ or $\Delta\ge 5$.

In Section~\ref{sec:girth}, we provide two infinite families
of $3$-connected planar graphs of girth $4$ with maximum degree $\Delta = 3$ and maximum degree $\Delta = 4$, for which the injective chromatic number is sharp for a conjecture by La and \v Storgel~\cite{la20232}. We also provide an infinite family of $3$-connected planar graphs of girth~$5$ with maximum degree $\Delta = 3$.




\subsection{Notation}
For a vertex $v\in V(G)$, we denote by $N_G(v)$ the set of its neighbours and $N_G[v] := N_G(v)\cup \{v\}$. Let $V\subset{V(G)}$, then $N_G[V] = \bigcup_{v\in V} N_G[v]$. We sometimes drop the subscript $G$, when it is clear from the context which graph is meant.
We denote the maximum degree of a graph by $\Delta(G)$.

A \emph{plane} graph is a planar graph embedded in the Euclidean plane. A triangle $uvw$ in a plane graph is a \emph{facial} triangle if it is the boundary of a face. A famous theorem by Whitney~\cite{Wh33}
states that $3$-connected planar graphs have a unique embedding. Therefore, it makes sense to call a triangle of a planar graph facial if the graph is $3$-connected.

\section{Computational Results} \label{sect:results}


In order to verify Conjecture~\ref{conj:luzar} for small graphs, we implemented a backtracking algorithm for determining the injective chromatic number of a given graph. Starting by choosing a vertex of maximum degree and giving all of its neighbours a different colour, we recursively add all possible colours to an uncoloured vertex.
This vertex $v$ is chosen such that the number of colours used by the vertices sharing a common neighbour with $v$ is maximum. (This is an injective variant of the well-known saturation degree used for the classical chromatic number.) We prune the search whenever we use at least as many colours as are used in a previously found injective colouring of the graph.

We performed various tests to verify the correctness of our implementation, see Appendix~\ref{app:correctness}.
Our implementation of this algorithm is open source and can be found on GitHub~\cite{DGR24Program}.

For every order $n$, we generated all connected graphs of minimum degree $2$ with at most $3n-6$ edges using the generator \texttt{geng} and filtered the planar ones using \texttt{planarg}, two utilities part of McKay's \texttt{nauty} package~\cite{Mc14}. This is one of the best existing methods for generating planar graphs of general connectivity. For each of these graphs, we then determined their injective chromatic number using our algorithm.

It is not necessary to consider vertices of degree $1$, since if a graph $G$ with a pendant edge $e$ incident to vertex $w$ of degree $1$ would be a counterexample to Conjecture~\ref{conj:luzar}, then $G - w$
is a counterexample since $\chi_i(G-w) \geq \chi_i(G) - 1$ and $\Delta(G-w) \geq \Delta(G) - 1$ and $\chi_i(G-w) = \chi_i(G) - 1$ can only happen if one of the endpoints of $e$ has degree $\Delta(G) = \chi_i(G)$ in $G$, hence $\Delta(G-w)=\Delta(G) - 1$.

We obtain the following results.

\begin{proposition}
    Conjecture~\ref{conj:luzar} holds up to order $13$.
\end{proposition}
For the computation of order $13$, the generation took approximately $1.61$ CPU years, while our injective chromatic number computation took only $15$ CPU hours.

In Table~\ref{tab:sharp_for_LS}, we give counts of planar graphs with minimum degree $2$ attaining the upper bound of Conjecture~\ref{conj:luzar} for each order $n$ and maximum degree $\Delta$. The smallest most symmetric examples for each $\Delta$ in Table~\ref{tab:sharp_for_LS} can among others be inspected on the House of Graphs~\cite{CDG23} by searching for the keywords ``injective chromatic number''. See Appendix~\ref{app:hog} for their specific URLs.

\begin{table}[!htb]
    \centering
    \begin{tabular}{c|*{10}{c}}
        \backslashbox{$n$}{$\Delta$} & 3 & 4 & 5 & 6 & 7 & 8 & 9 & 10 & 11 & 12   \\
        \hline
        3 -- 8 & 0&0&0&0&0&0&0&0&0&0\\
        9 & 1&1&0&0&0&0&0&0&0&0\\
        10 &3&0&1&0&0&0&0&0&0&0\\
        11 &2&0&2&36&0&0&0&0&0&0\\
        12 &1&0&8&326&131&0&0&0&0&0\\
        13 &11&0&30&3371&2600&887&0&0&0&0\\
    \end{tabular}
    \caption{The number of connected planar graphs with minimum degree $2$ attaining the upper bound of Conjecture~\ref{conj:luzar} for each order $n$ and maximum degree $\Delta$.}
    \label{tab:sharp_for_LS}
\end{table}

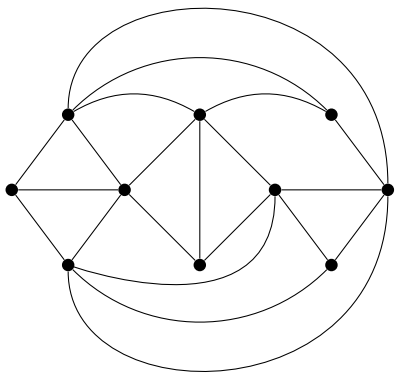
\begin{figure}[!htb]
    \centering
    \begin{tikzpicture}[vertex/.style={circle,fill, scale = 0.5}, scale = 1]
        \node[vertex] (1) at (-2.5,0) {};
        \node[vertex] (2) at (-1,0) {};
        \node[vertex] (3) at (1,0) {};
        \node[vertex] (4) at (2.5,0) {};

        \node[vertex] (5) at (-1.75,1) {};
        \node[vertex] (6) at (0,1) {};
        \node[vertex] (7) at (1.75,1) {};

        \node[vertex] (8) at (-1.75,-1) {};
        \node[vertex] (9) at (0,-1) {};
        \node[vertex] (10) at (1.75,-1) {};

        \draw (2) -- (8) -- (1) -- (2) -- (5) -- (1);
        \draw (2) -- (6) -- (3) -- (9) -- (2);
        \draw (6) -- (9);
        \draw (4) -- (3) -- (10) -- (4) -- (7);

        \draw (7) to[bend right] (6);
        \draw (6) to[bend right] (5);
        \draw (5) to[bend left, out = 45, in=135] (7);
        \draw (5) .. controls +(0,2) and +(0,3) .. (4);
        \draw (8) to[bend right, out = -45, in = -135] (10);
        \draw (8.center) .. controls +(0,0) and +(0,-2) .. (3);
        \draw (8) .. controls +(0,-2) and +(0,-3) .. (4);

    \end{tikzpicture}
    \caption{Planar graph of order $10$ with maximum degree $5$, diameter $2$ and injective chromatic number $10$.}
    \label{fig:LS_correction}
\end{figure}

While analysing the results from Table~\ref{tab:sharp_for_LS} we noticed that the single graph of order $10$ with maximum degree $5$ and injective chromatic number $10$ we found is different from the one provided by Lu\v{z}ar and \v{S}krekovski in Figure~3 of~\cite{luvzar2015counterexamples} as a counterexample to Conjecture~\ref{conj:chen} for $\Delta = 5$. Theirs was not a graph of diameter $2$ and injective chromatic number $10$, as they claimed.
Moreover, it is also not a subgraph of a planar graph of diameter $2$, maximum degree $5$ and order $10$. However, it is still a counterexample as its injective chromatic number is 9. Using our algorithm, we correct their claim by giving a  planar graph of order 10 with diameter $2$ and maximum degree $5$ in which every edge is part of a triangle. This graph can be found in Figure~\ref{fig:LS_correction}. Hence, the results by Lu\v{z}ar and \v{S}krekovski~\cite{luvzar2015counterexamples} still hold.

Interestingly enough, it seems that for $\Delta = 4$, we only found exactly one planar graph\footnote{This graph can be inspected on the House of Graphs~\cite{CDG23} at \url{https://houseofgraphs.org/graphs/33503}.}
$G$ of minimum degree $2$ for which $\chi_i(G) = 9$. See Figure~\ref{fig:D4_chi9}.
This raises the question whether other such graphs exist. As the graph is quartic it will be unsuitable for the operations we describe in Section~\ref{sec:families}.

\begin{figure}[!htb]
\centering
\begin{tikzpicture}[main_node/.style={circle,fill, scale = 0.5}, scale = 0.7]

\node[main_node] (0) at (3.1428571428571423, 0.8571428571428563) {};
\node[main_node] (1) at (-0.8571428571428563, -3.1428571428571423) {};
\node[main_node] (2) at (0.3843973995665939, -0.3610601766961157) {};
\node[main_node] (3) at (-0.8571428571428563, 4.857142857142858) {};
\node[main_node] (4) at (0.3843973995665939, 2.073012168694783) {};
\node[main_node] (5) at (-2.0986831138523083, -0.3610601766961157) {};
\node[main_node] (6) at (-4.857142857142858, 0.8571428571428563) {};
\node[main_node] (7) at (-0.8571428571428563, 0.8571428571428563) {};
\node[main_node] (8) at (-2.0986831138523083, 2.07534589098183) {};

 \path[draw]
(0) edge node {} (1)
(0) edge node {} (2)
(0) edge node {} (3)
(0) edge node {} (4)
(1) edge node {} (2)
(1) edge node {} (5)
(1) edge node {} (6)
(2) edge node {} (5)
(2) edge node {} (7)
(3) edge node {} (4)
(3) edge node {} (6)
(3) edge node {} (8)
(4) edge node {} (7)
(4) edge node {} (8)
(5) edge node {} (6)
(5) edge node {} (7)
(6) edge node {} (8)
(7) edge node {} (8)
;

\end{tikzpicture}
\caption{Planar graph with maximum degree $4$ and injective chromatic number $9$.}\label{fig:D4_chi9}
\end{figure}
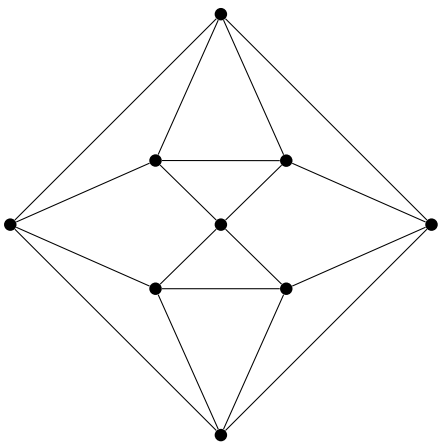


We can also generate planar graphs using \texttt{plantri}~\cite{brinkmann2007program}, which significantly outperforms \texttt{geng} and \texttt{planarg} when generating $k$-connected graphs for $k\in\{3,4\}$. We therefore used \texttt{plantri} to generate $3$-connected and $4$-connected planar graphs and obtained the following results.
\begin{proposition}
    A $3$-connected counterexample to Conjecture~\ref{conj:luzar} has order at least $17$.
    A $4$-connected counterexample to Conjecture~\ref{conj:luzar} has order at least $21$.
\end{proposition}
The computation for $3$-connected planar graphs of order $16$ took approximately 131 CPU hours for the generation and 2109 CPU hours for computing the injective chromatic number of the generated graphs. For $4$-connected planar graphs of order $20$, the generation took approximately 70 CPU hours while determining the injective chromatic number took 2817 hours.

In the hope of finding examples similar to that of Figure~\ref{fig:D4_chi9}, we can also use \texttt{plantri}~\cite{brinkmann2007program} to efficiently generate $3$-connected quartic planar graphs. We find the following.
\begin{proposition}
    The only $3$-connected quartic planar graph with injective chromatic number $9$ and order $n\leq 28$ is the one from Figure~\ref{fig:D4_chi9}.
\end{proposition}
The computation for order $28$ took approximately 2 CPU hours for the generation and 3024 hours for determining the injective chromatic number.

Recently, Jin and Xu~\cite{JX20} showed that a cubic counterexample of minimum order to Conjecture~\ref{conj:chen} and~\ref{conj:luzar} -- if it exists -- will be $3$-connected.
We therefore also generated $3$-connected planar cubic graphs using the generator \texttt{plantri}~\cite{brinkmann2007program}, determined their injective chromatic number using our algorithm and obtained the following result. Note that in this case we use \texttt{plantri} as it is a lot faster than \texttt{geng} and \texttt{planarg} at generating $3$-connected (cubic) planar graphs.
\begin{proposition}
    A cubic counterexample to Conjecture~\ref{conj:luzar} has order at least $38$.
\end{proposition}
The computation on order $36$ took approximately 50 hours for the generation and 2275 hours for determining the injective chromatic number of the generated graphs.


\section{Infinite families of \texorpdfstring{$\bm3$}{3}-connected planar graphs for each maximum degree}\label{sec:families}
Lu\v{z}ar and \v{S}krekovski~\cite{luvzar2015counterexamples} provided only finitely many counterexamples to Conjecture~\ref{conj:chen} for every maximum degree $\Delta$.
Their counterexamples were all graphs with diameter $2$ and one can show that for a fixed maximum degree $\Delta$, the order of these graphs is bounded by $1+\Delta^2$. Hence, it is natural to wonder whether these are the only counterexamples to Conjecture~\ref{conj:chen} or if we can save the conjecture by restricting the conjecture to graphs with diameter at least $3$, only ruling out a finite number of graphs for each maximum degree $\Delta$.

It turns out that even when considering only graphs of diameter at least $3$, Conjecture~\ref{conj:chen} is false. In this section we present several infinite families of $3$-connected planar graphs for each maximum degree $\Delta$ with a certain injective chromatic number. Hence, these families contain graphs of diameter at least $3$.
Note that such families of connectivity $1$ and connectivity $2$ are often easily obtainable by attaching pendant edges or subdividing paths in the correct way.

We first present an infinite family of counterexamples to Conjecture~\ref{conj:chen} when $\Delta = 4$.




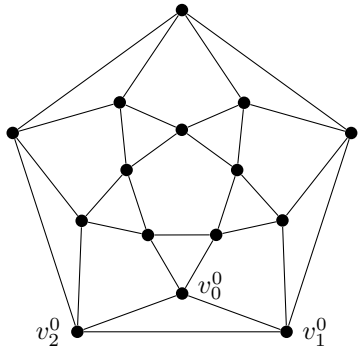
\begin{figure}[!htb]
\centering
\begin{tikzpicture}[main_node/.style={circle,draw,minimum size=1em,inner sep=3pt, scale = 0.7]}, scale = 0.045]

\node [circle,fill,scale=0.5] (15) at (68.343085,70.154872) {};
\node [circle,fill,scale=0.5, label={[yshift=3pt]right:$v_0^0$}] (14) at (50.057916,13.760642) {};
\node [circle,fill,scale=0.5] (13) at (39.952692,31.056590) {};
\node [circle,fill,scale=0.5] (12) at (50.005848,97.563304) {};
\node [circle,fill,scale=0.5] (11) at (49.956039,62.160366) {};
\node [circle,fill,scale=0.5, label=left:$v_2^0$] (10) at (19.100205,2.439906) {};
\node [circle,fill,scale=0.5] (9) at (31.621097,70.269455) {};
\node [circle,fill,scale=0.5] (8) at (33.644443,50.281644) {};
\node [circle,fill,scale=0.5] (7) at (0.000000,61.231928) {};
\node [circle,fill,scale=0.5] (6) at (20.372455,35.246865) {};
\node [circle,fill,scale=0.5] (5) at (60.107320,31.099592) {};
\node [circle,fill,scale=0.5] (4) at (66.313844,50.264391) {};
\node [circle,fill,scale=0.5] (3) at (99.999999,61.217047) {};
\node [circle,fill,scale=0.5, label=right:$v_1^0$] (2) at (80.904481,2.436695) {};
\node [circle,fill,scale=0.5] (1) at (79.681727,35.366557) {};
\draw [black] (15) to (4);
\draw [black] (15) to (11);
\draw [black] (15) to (12);
\draw [black] (15) to (3);
\draw [black] (14) to (2);
\draw [black] (14) to (10);
\draw [black] (14) to (13);
\draw [black] (14) to (5);
\draw [black] (13) to (6);
\draw [black] (13) to (8);
\draw [black] (13) to (5);
\draw [black] (12) to (3);
\draw [black] (12) to (9);
\draw [black] (12) to (7);
\draw [black] (11) to (4);
\draw [black] (11) to (8);
\draw [black] (11) to (9);
\draw [black] (10) to (2);
\draw [black] (10) to (7);
\draw [black] (10) to (6);
\draw [black] (9) to (8);
\draw [black] (9) to (7);
\draw [black] (8) to (6);
\draw [black] (7) to (6);
\draw [black] (5) to (1);
\draw [black] (5) to (4);
\draw [black] (4) to (1);
\draw [black] (3) to (1);
\draw [black] (3) to (2);
\draw [black] (2) to (1);

\end{tikzpicture}
\caption{Graph $G^4_0$ with $\Delta = 4$ and $\chi_i(G^4_0) = 8$.}
\label{fig:deg_4_inj_chr_num_8}
\end{figure}

We inductively define the following family $\mathcal{G}$. Let $G^4_0$ be the graph\footnote{This graph can be inspected on the House of Graphs~\cite{CDG23} at \url{https://houseofgraphs.org/graphs/50484}.}
shown in Figure~\ref{fig:deg_4_inj_chr_num_8}
and let $v_0^0v_1^0v_2^0$ be the vertices of a triangle adjacent with the outer face. Given $G^4_{i-1}$ with $i > 0$, we define $G^4_i$ to be the graph with vertex set $V(G^4_{i-1})\cup \{v_0^i,v_1^i,v_2^i\}$ and edge set $E\cup E(G^4_{i-1})\setminus \{v_0^{i-1}v_1^{i-1},v_1^{i-1}v_2^{i-1},v_2^{i-1}v_0^{i-1}\}$, where
\[E := \{v_0^{i}v_1^{i},v_1^{i}v_2^{i},v_2^{i}v_0^{i}\}\cup\{ v_0^iv_1^{i-1}, v_0^iv_2^{i-1}, v_1^iv_0^{i-1}, v_1^iv_2^{i-1}, v_2^iv_0^{i-1}, v_2^iv_1^{i-1}\}\]
Then $\mathcal{G}:=\{G^4_i\mid i \ge 0\}$. It is straightforward to see that every graph in $\mathcal{G}$ is planar, has maximum degree~$4$ and one can show using Menger's theorem that they are all $3$-connected.


\begin{theorem}
For every graph $G \in \mathcal{G}$, $\chi_i(G) = 8$.
\label{theor:pentagonFamily}
\end{theorem}

\begin{proof}
    We will first show that $\chi_i(G_i^4) \le 8$ in the case when $i$ is even.

    It is straightforward to verify that $\chi_i(G^4_0) \leq 8$ and that any injective $8$-colouring uses three different colours for $v_0^0, v_1^0$ and $v_2^0$. Let $i\geq 2$ even and suppose that $G^4_{i-2}$ has an injective $8$-colouring $c_{i-2}$ in which $v^{i-2}_0, v^{i-2}_1, v^{i-2}_2$ respectively have colours $0,1,2$. We define the following $8$-colouring $c_i$ on $G^4_{i}$. If $v\in V(G^4_i)$ is a vertex of $G^4_{i-2}$ let $c_i(v) := c_{i-2}(v)$. Let $c_i(v^{i-1}_j) := c_{i-2}(v^{i-2}_j)$ for $j\in\{0,1,2\}$ and let $v^{i}_0, v^i_1, v^i_2$ have three different colours pairwise disjoint from $\{0,1,2\}$.

    We show that $c_i$ is an injective colouring. Indeed, $v^{i}_0$ only has $v^{i}_1, v^i_2$ and vertices with colours $0,1$ or $2$ at distance $2$. A similar reasoning holds for $v^i_1$ and $v^i_2$. Vertex $v^{i-1}_0$ only has $v^i_0, v^i_1, v^i_2$ and the neighbours of $v^{i-2}_1$ and $v^{i-2}_2$ at distance at most $2$. By definition $v^{i-1}_1$ and $v^{i-1}_2$ have different colours than $v^{i-1}_0$.
    So without loss of generality assume that a $v\in N_{G^4_i}(v^{i-2}_1)\setminus \{v^{i-1}_0, v^{i-1}_2\}$ has the same colour as $v^{i-1}_0$. Then $c_{i-2}(v) = c_{i}(v) = c_i(v^{i-1}_0) = c_{i-2}(v^{i-2}_0)$. This is a contradiction, since $v$ and $v^{i-2}_0$ have a path of length $2$ in $G^4_{i-2}$. A similar reasoning holds for $v^{i-1}_1$ and $v^{i-1}_2$. If $c_i$ were not injective for another pair of vertices, then the same holds for $c_{i-2}$, which shows that $c_i$ is an injective $8$-colouring.

    We now show that $\chi_i(G_i^4) \le 8$ in the case when $i$ is odd.
    
    Again, one can verify that $\chi_i(G^4_1)\leq 8$ and that any injective $8$-colouring will use six different colours for $v^0_0,v^0_1,v^0_2,v^1_0,v^1_1$ and $v^1_2$.
    Let $i\geq 3$ odd and suppose that $G^4_{i-2}$ has an injective $8$-colouring $c_{i-2}$ in which $v^{i-3}_0,v^{i-3}_1,v^{i-3}_2,v^{i-2}_0,v^{i-2}_1, v^{i-2}_2$ have different colours.
    We define an $8$-colouring $c_{i}$ of $G^4_i$. If $v\in V(G^4_i)$ is a vertex of $G^4_{i-2}$ let $c_i(v) := c_{i-2}(v)$. Let $c_i(v^{i-1}_j) := c_{i-2}(v^{i-2}_j)$ and $c_i(v^{i}_j) := c_{i-2}(v^{i-3}_j)$ for $j\in\{0,1,2\}$.

    We show that $c_i$ is an injective colouring. By definition, $v^i_{0}, v^{i}_1, v^i_2$ do not share a colour with any vertex at distance at most $2$. Vertex $v^{i-1}_0$ only has paths of length $2$ to $v^{k}_j$ with $j\in\{0,1,2\}$ and $k\in\{i-3,i-2,i-1,i\}$ except when $v^k_j = v^{i-1}_0$ and $v^k_j = v^{i-2}_0$. We have that $c_i(v^k_j) = c_i(v^{i-1}_0)$ when $v^k_j = v^{i-1}_0$ and $v^k_j = v^{i-2}_0$. A similar reasoning holds for $v^{i-1}_1$ and $v^{i-1}_2$. If $c_i$ were not injective for another pair of vertices, then the same holds for $c_{i-2}$, which shows that $c_i$ is an injective $8$-colouring.

    Finally, we show that $\chi_i(G^4_i)>7$ for any $i \geq 0$.

    Via computer verification it is easy to see that $\chi_i(G^4_0) >7$ and $\chi_i(G^4_1 - v_0^1v_1^1 - v_0^1v_2^1 - v_1^1v_2^1) > 7$. A direct proof of this fact can be found in Appendix~\ref{app:two_graphs_inj_at_least_7}.
    Since $G^4_1 - v_0^1v_1^1 - v_0^1v_2^1 - v_1^1v_2^1$ is a subgraph of any $G\in \mathcal{G}$ except for $G^4_0$, we conclude that $\chi_i(G) = 8$ for all $G\in \mathcal{G}$.
\end{proof}


\begin{figure}[!htb]
    \centering
    \begin{tikzpicture}[vertex/.style={circle,fill,inner sep=3pt, scale=0.7}]
        \node[draw,minimum size=3cm,regular polygon,regular polygon sides=3] (a) {};

       \node[vertex, label=$v_0^{i-1}$] (v0) at (a.corner 1) {};
       \node[vertex, label=below:$v_1^{i-1}$] (v1) at (a.corner 2) {};
       \node[vertex, label=below:$v_2^{i-1}$] (v2) at (a.corner 3) {};
       \node[label=below:$G^\Delta_{i-1}$] at (0,-1) {};
    \end{tikzpicture}
    \begin{tikzpicture}[vertex/.style={circle,fill,inner sep=3pt, scale=0.7}]
        \node[draw,minimum size=3cm,regular polygon,regular polygon sides=3] (a) {};
       \node[vertex, label=$v_0^{i-1}$] (v0) at (a.corner 1) {};
       \node[vertex, label=below:$v_1^{i-1}$] (v1) at (a.corner 2) {};
       \node[vertex, label=below:$v_2^{i-1}$] (v2) at (a.corner 3) {};
        \node[vertex, label=left:$v_2^{i}$] (v2i) at ($(v0)!0.5!(v1)$) {};
        \draw (v2) to (v2i);
       \node[label=below:$G^{'\Delta}_{i}$] at (0,-1) {};
    \end{tikzpicture}
    \begin{tikzpicture}[vertex/.style={circle,fill,inner sep=3pt, scale=0.7}]
        \node[draw,minimum size=3cm,regular polygon,regular polygon sides=3] (a) {};
       \node[vertex, label=$v_0^{i-1}$] (v0) at (a.corner 1) {};
       \node[vertex, label=below:$v_1^{i-1}$] (v1) at (a.corner 2) {};
       \node[vertex, label=below:$v_2^{i-1}$] (v2) at (a.corner 3) {};
        \node[vertex, label=left:$v_2^{i}$] (v2i) at ($(v0)!0.33!(v1)$) {};
        \draw (v2) to (v2i);
        \node[vertex, label=left:$v_1^{i}$] (v1i) at ($(v0)!0.66!(v1)$) {};
        \draw (v2) to (v1i);
       \node[label=below:$G^{\Delta}_{i}$] at (0,-1) {};
    \end{tikzpicture}
    \caption{Visualisation of how to construct $G^{'\Delta}_i$ and $G^\Delta_i$ from $G^\Delta_{i-1}$ for a given triangle $v_0^{i-1}v_1^{i-1}v_2^{i-1}$ and $i\geq 1, \Delta\geq 5$.}
    \label{fig:general_delta_construction}
\end{figure}
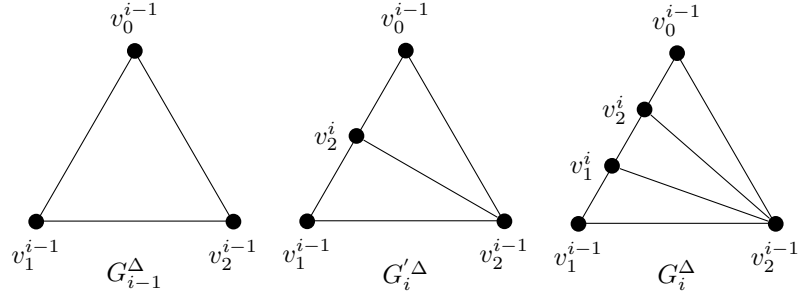

\noindent For $\Delta\geq 5$, we use the following procedure for constructing $3$-connected infinite families.

Let $G^\Delta_0$ be a graph and let $v^0_0v^0_1v^0_2$ be a triangle of $G^\Delta_0$. For $i>0$ we inductively define $G^{'\Delta}_i$ from $G^\Delta_{i-1}$ by subdividing $v^{i-1}_0v^{i-1}_1$ with vertex $v^i_2$ and adding edge $v^{i-1}_2v^i_2$ and we define $G^\Delta_i$ from $G^{'\Delta}_i$ by subdividing $v^i_2v^{i-1}_1$ with vertex $v^i_1$, adding edge $v^{i-1}_2v^i_1$ and we define $v^i_0:= v^{i-1}_2$. See Figure~\ref{fig:general_delta_construction}.
We define the infinite family $\mathcal{H}(G^\Delta_0; v^0_0v^0_1; v^0_2) := \{G^\Delta_0\}\cup\{G^{'\Delta}_i, G^\Delta_i\mid i > 0\}$. It is easy to see that if $G^\Delta_0$ is a plane graph such that $v^0_0v^1_0v^2_0$ is a facial triangle, every $G^\Delta_i$ and $G^{'\Delta}_i$ are planar.

\begin{lemma}\label{lemma:inf_family_procedure}
    Let $G$ be a $3$-connected planar graph of order at least $8$, maximum degree $\Delta \geq 5$ and diameter $2$ such that every edge is part of a triangle. Let $uv$ be an edge lying on a facial triangle $uvw$ such that the degree of $w$ is at most $\Delta - 2$ and at most $\lvert V(G)\rvert - 5$
    and $\lvert N[\{u,v,w\}]\rvert \leq \lvert V(G) \rvert - 2$. If $(u,w)$ is the only pair of vertices in $G$ in which there is a unique path of length $2$ containing $uv$ between them or if there are no such pairs,
    then $\mathcal{H}:=\mathcal{H}(G; uv; w)$ is an infinite family of $3$-connected planar graphs with maximum degree $\Delta$ and $\chi_i(H) = \chi_i(G)$ for all $H\in\mathcal{H}$.
\end{lemma}
\begin{proof}
    Let $G^\Delta_0:= G$, $v_0^0:= u$, $v_1^0:= v$ and $v_2^0 := w$. We will use the same notation as in the definition of the procedure.

    We first show that all graphs in the family have maximum degree $\Delta$.
    We have that the maximum degree of $G^\Delta_0$ is $\Delta$ and that the degree of $v_2^0$ is at most $\Delta - 2$. Suppose that $i > 0$ and that $v_2^{i-1}$ has degree at most $\Delta - 2$ in $G_{i-1}$, then in $G^{'\Delta}_i$, we have incremented the degree of $v_2^{i-1}$ by one, $v^{i}_2$ has degree $3$ and the degrees of the remaining vertices remain unchanged. In $G^\Delta_i$, we have again incremented the degree of $v_2^{i-1}$ by one, $v^{i}_1$ has degree $3$ and the degrees of the remaining vertices remain unchanged. Hence, $\Delta(G^{'\Delta}_i) = \Delta(G^\Delta_i) = \Delta$ and the degree of $v^i_2$ is $3$ which is at most $\Delta - 2$.

    Now we show that every graph in the family has the same injective chromatic injective number as $G^\Delta_0$. It is easy to see that $\chi_i(G^\Delta_0) = \lvert V(G^\Delta_0)\rvert$, since its diameter is $2$ and every edge is part of a triangle. By definition of $u,v$ and $w$, we see that any distinct pair $x,y\in V(G^\Delta_0)$ has a path of distance $2$ between them in $G^{'\Delta}_1$ and $G^\Delta_1-v_2^0v^1_1$. Therefore, we need at least $\lvert V(G^\Delta_0)\rvert$ colours to obtain an injective colouring of these graphs. Since $G^\Delta_1-v_2^0v^1_1$ is a subgraph of $G^\Delta_i$ and $G^{'\Delta}_j$ for $i \ge 1$, $j \ge 2$, we have that the injective chromatic number is at least $\lvert V(G^\Delta_0)\lvert$ for any graph in $\mathcal{H}$.

    Now let $i\ge 0$ and $c_{i}$ be an injective $\lvert V(G^\Delta_0)\rvert$-colouring of $G^\Delta_{i}$ and define $N_i := N_{G^\Delta_{i}}[\{v^i_0, v^i_1, v^i_2\}]$. By our assumptions, if $i=0$, we have that $\lvert N_0\rvert\leq \lvert V(G^\Delta_0)\rvert - 2$. Since the degree of $v_2^0$ is at most $\lvert V(G^\Delta_0)\rvert - 5$ we also have that $\lvert N_1\rvert \leq \lvert V(G^\Delta_0)\rvert - 2$.  Arguing in the same way as in the start of this proof, we see that for $i \ge 2$, $\lvert N_i \rvert = 6 \leq \lvert G^\Delta_0\rvert - 2$. Now define $c'_{i+1}$ and $c_{i+1}$ to be colourings of $G^{'\Delta}_{i+1}$ and $G^\Delta_{i+1}$, respectively, such that $c'_{i+1}(v) = c_{i+1}(v) = c_i(v)$ if $v\in V(G^\Delta_i)$, $c'_{i+1}(v^{i+1}_2)$ gets a colour not assigned to any vertex of $N_i$ and $c_{i+1}(v^{i+1}_2)$ and $c_{i+1}(v^{i+1}_1)$ get distinct colours not assigned to any vertex of $N_i$. Then these are both injective $\lvert V(G^\Delta_0)\rvert$-colourings of $G^{'\Delta}_{i+1}$ and $G^\Delta_{i+1}$, respectively. Hence, the injective chromatic number is at most $\lvert V(G^\Delta_0)\rvert$ for any graph in $\mathcal{H}$.

    Finally, we show that the graphs are $3$-connected by induction using Menger's theorem.
    For a graph $H$ which is $3$-connected there are three internally disjoint paths between any pair of vertices. For any triangle $uvw$ and a vertex $x$, we also have three internally disjoint paths from $x$ to $u, v$ and $w$. This follows from Menger's theorem applied to $H$ with an added vertex adjacent to $u$, $v$ and $w$. If we subdivide $uv$ by $w'$ and add an edge $ww'$, all these paths still exist in the new graph $H'$, but are subdivided if they contain the edge $uv$.
    There exist three paths from $w'$ to $u$, to $v$ and to $w$, which are internally disjoint, so for any $x\in V(H)$ such that $x\neq u, v, w$ there exist three internally disjoint paths between $x$ and $w'$. It is easy to see three such paths also exist between $w$ and $w'$. Finally, let $x = u$ or $x = v$, then we have the paths $w'x$ and $w'wx$. Let $z\in V(H)$ with $z\neq u,v,w$, then there exists a path $P_1$ between $x$ and $z$ and a path $P_2$ between $z$ and $y$, where $y\in \{u,v\}\setminus\{x\}$, such that they are internally disjoint and neither contains $w$. Then $P_1P_2yw'$ is a path internally disjoint from $w'x$ and $w'wx$. Hence, this new graph is also $3$-connected by Menger's theorem.
    Since all graphs are inductively obtained by applying a subdivision to the edge of a triangle and adding an edge, by induction it follows that all graphs in $\mathcal{H}$ are $3$-connected.
\end{proof}







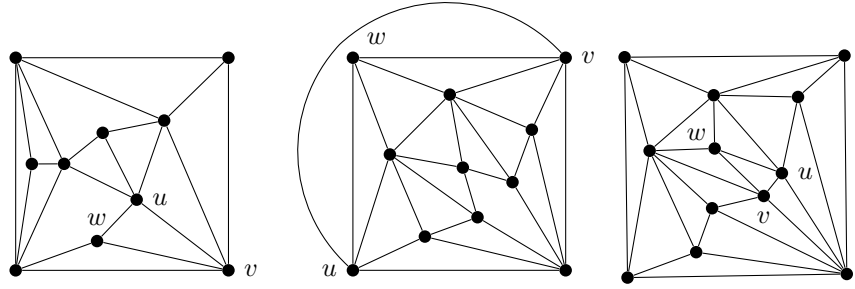
\begin{figure}
    \centering
    \begin{tikzpicture}[main_node/.style={circle,fill, scale = 0.5}, scale = 0.4, rotate=135]
\node[main_node] (0) at (5.142857142857142, -0.1442659904198358) {};
\node[main_node] (1) at (0.14285714285714235, -5.142857142857142) {};
\node[main_node] (2) at (0.14285714285714235, 4.857142857142858) {};
\node[main_node] (3) at (2.2651795780318906, 1.9746407438715128) {};
\node[main_node] (4) at (1.5098244483813836, 1.2138630600169065) {};
\node[main_node] (5) at (0.17667901433403088, -2.1561003099464635) {};
\node[main_node, label=right:$v$] (6) at (-4.857142857142858, -0.1442659904198358) {};
\node[main_node, label=above:$w$] (7) at (-1.0803672088903207, 2.2564102564102573) {};
\node[main_node, label=right:$u$] (8) at (-1.03808986954421, 0.35164835164835084) {};
\node[main_node] (9) at (1.332259623127717, -0.42321780783319163) {};

\path[draw]
(0) edge node {} (1) 
(0) edge node {} (2) 
(0) edge node {} (3) 
(0) edge node {} (4) 
(0) edge node {} (5) 
(1) edge node {} (5) 
(1) edge node {} (6) 
(2) edge node {} (3) 
(2) edge node {} (4) 
(2) edge node {} (6) 
(2) edge node {} (7) 
(3) edge node {} (4) 
(4) edge node {} (8) 
(4) edge node {} (9) 
(5) edge node {} (6) 
(5) edge node {} (8) 
(5) edge node {} (9) 
(6) edge node {} (7) 
(6) edge node {} (8) 
(7) edge node {} (8) 
(8) edge node {} (9) 
;

\end{tikzpicture}
        \begin{tikzpicture}[main_node/.style={circle,fill, scale = 0.5}, scale = 0.5, rotate=-45]

    \node[main_node] (0) at (3.1428571428571423, 0.7857142857142851) {};
    \node[main_node, label=left:$u$] (1) at (-0.8571428571428563, -3.2142857142857135) {};
    \node[main_node, label=right:$v$] (2) at (-0.8571428571428563, 4.785714285714286) {};
    \node[main_node] (3) at (-0.1383762436394016, 2.7952836637047165) {};
    \node[main_node] (4) at (0.4846965899597482, 1.4449381028328396) {};
    \node[main_node] (5) at (0.48682311840206527, 0.12649046859573243) {};
    \node[main_node] (6) at (-0.1383762436394016, -1.2259816207184624) {};
    \node[main_node] (7) at (-2.3414597098807626, -0.3392192602718911) {};
    \node[main_node, label = above right:$w$] (8) at (-4.857142857142858, 0.7857142857142851) {};
    \node[main_node] (9) at (-2.3414597098807626, 1.9127743601427816) {};
    \node[main_node] (10) at (-0.7231715652768278, 0.7857142857142851) {};

    \node[] (t) at (-5.357142857142858, 0.7857142857142851) {};

     \path[draw]
    (0) edge node {} (1) 
    (0) edge node {} (2) 
    (0) edge node {} (3) 
    (0) edge node {} (4) 
    (0) edge node {} (5) 
    (0) edge node {} (6) 
    (1) edge node {} (6) 
    (1) edge node {} (7) 
    (1) edge node {} (8) 
    (2) edge node {} (3) 
    (2) edge node {} (8) 
    (2) edge node {} (9) 
    (3) edge node {} (4) 
    (3) edge node {} (9) 
    (4) edge node {} (9) 
    (4) edge node {} (10) 
    (5) edge node {} (6) 
    (5) edge node {} (7) 
    (5) edge node {} (10) 
    (6) edge node {} (7) 
    (7) edge node {} (8) 
    (7) edge node {} (9) 
    (7) edge node {} (10) 
    (8) edge node {} (9) 
    (9) edge node {} (10) 
    ;

    \draw[] (1) to[bend left = 45] (t.center) to[bend left=45] (2);

    \end{tikzpicture}
        \begin{tikzpicture}[main_node/.style={circle,fill, scale = 0.5}, scale = 0.43, rotate=-45
]

    \node[main_node] (0) at (4.742857142857144, 0.12857142857142811) {};
    \node[main_node] (1) at (0.02857142857142847, -4.742857142857144) {};
    \node[main_node] (2) at (-0.08571428571428541, 4.857142857142858) {};
    \node[main_node] (3) at (-0.1961981566820281, 2.9357718894009217) {};
    \node[main_node, label=right:$u$] (4) at (1.119124423963134, 0.9158410138248838) {};
    \node[main_node, label=below:$v$] (5) at (1.221774193548387, 0.020391705069124022) {};
    \node[main_node] (6) at (0.36065668202765, -1.3771313364055295) {};
    \node[main_node] (7) at (0.9735599078341011, -2.6861755511727745) {};
    \node[main_node] (8) at (-2.2621543778801834, -1.4861751152073734) {};
    \node[main_node] (9) at (-4.857142857142858, 0.014285714285713347) {};
    \node[main_node] (10) at (-2.0767857142857142, 1.1307027649769583) {};
    \node[main_node, label={[label distance=-3pt]above left:$w$},] (11) at (-0.8917626728110601, -0.013824884792627223) {};
    
     \path[draw]
    (0) edge node {} (1) 
    (0) edge node {} (2) 
    (0) edge node {} (3) 
    (0) edge node {} (4) 
    (0) edge node {} (5) 
    (0) edge node {} (6) 
    (0) edge node {} (7) 
    (1) edge node {} (7) 
    (1) edge node {} (8) 
    (1) edge node {} (9) 
    (2) edge node {} (3) 
    (2) edge node {} (9) 
    (2) edge node {} (10) 
    (3) edge node {} (4) 
    (3) edge node {} (10) 
    (4) edge node {} (5) 
    (4) edge node {} (10) 
    (4) edge node {} (11) 
    (5) edge node {} (6) 
    (5) edge node {} (8) 
    (5) edge node {} (11) 
    (6) edge node {} (7) 
    (6) edge node {} (8) 
    (7) edge node {} (8) 
    (8) edge node {} (9) 
    (8) edge node {} (10) 
    (8) edge node {} (11) 
    (9) edge node {} (10) 
    (10) edge node {} (11) 
    ;
    
    \end{tikzpicture}
    \caption{Base graphs for the infinite families constructed in Theorem~\ref{thm:inf_families}, satisfying the requirements of Lemma~\ref{lemma:inf_family_procedure} for the triangle $uvw$. Their maximum degree is $5$, $6$ and $7$, respectively. Note that the left graph is the same graph as Figure~\ref{fig:LS_correction}.}
    \label{fig:base_graphs}
\end{figure}
\begin{theorem}\label{thm:inf_families}
    For every $\Delta > 4$, there is an infinite family of $3$-connected planar graphs with maximum degree $\Delta$ and injective chromatic number $\Delta + 5$, if $5\le \Delta \le 7$, or $\left\lfloor \frac{3}{2}\Delta\right\rfloor + 1$, if $\Delta \ge 8$.
\end{theorem}
\begin{proof}
    We prove this by finding a suitable starting graph for each $\Delta$ and using Lemma~\ref{lemma:inf_family_procedure}.

    For $\Delta = 5$, let $G^5$ be the left graph\footnote{This graph can be inspected on the House of Graphs~\cite{CDG23} at \url{https://houseofgraphs.org/graphs/52994}.}
    of Figure~\ref{fig:base_graphs} and let $u$, $v$, $w$ be the vertices indicated in the figure. It is straightforward to check that the graph has diameter $2$, that $(u,w)$ is the only pair of vertices which has a unique path of length $2$ between them containing $uv$ and that the other requirements of the lemma are fulfilled, hence $\mathcal{H}(G^5; uv; w)$ is an infinite family of $3$-connected planar graphs with maximum degree $5$ and injective chromatic number $10$.


    For $\Delta = 6$, let $G^6$ be the middle graph\footnote{This graph can be inspected on the House of Graphs~\cite{CDG23} at \url{https://houseofgraphs.org/graphs/52998}.}
    of Figure~\ref{fig:base_graphs} and let $u$, $v$, $w$ be the vertices indicated in the figure. Again one can check that the graph has diameter $2$. No pairs have a unique path of length $2$ between them containing $uv$, since $G^6-uv$ is also of diameter $2$. The other requirements of the lemma are also easily verified, hence $\mathcal{H}(G^6;uv;w)$ is an infinite family of $3$-connected planar graphs with maximum degree $6$ and injective chromatic number $11$.

    For $\Delta = 7$, let $G^7$ be the right graph\footnote{This graph can be inspected on the House of Graphs~\cite{CDG23} at \url{https://houseofgraphs.org/graphs/52997}.}
    of Figure~\ref{fig:base_graphs} and let $u$, $v$, $w$ be the vertices indicated in the figure. Again one can check that the graph has diameter $2$. No pairs have a unique path of length $2$ between them containing $uv$, since $G^7 - uv$ is also of diameter 2. The other requirements of the lemma are also easily verified, hence $\mathcal{H}(G^7;uv;w)$ is an infinite family of $3$-connected planar graphs with maximum degree $7$ and injective chromatic number $12$.

    For $\Delta \geq 8$, we use the diameter $2$ graphs given by Lu\v zar and \v Skrekovski in~\cite{luvzar2015counterexamples}, but add three edges $a_kb_k, a_kc_{k+1}$ and $b_kc_{k+1}$ for even $\Delta$ and $a_kb_k, a_kc_{k+2}$ and $b_kc_{k+2}$ for odd $\Delta$ to make them $3$-connected. Their construction is shown in Figure~\ref{fig:deg_+8}.
    Using the labellings of the figure, if $\Delta$ is even, let $u = c_{k+1}, v = b_k, w = a_k$. If $\Delta$ is odd, let $u = c_{k+2}, v = b_k, w = a_k$. Denote these graphs by $G^\Delta$. By construction, every $G^\Delta$ has diameter $2$. No pairs of vertices have a unique path of length $2$ between them containing $uv$, since removing $uv$ also gives a graph of diameter $2$. The other requirements of the lemma are also easily verified. Therefore, $\mathcal{H}(G^\Delta; uv; w)$ is an infinite family of $3$-connected planar graphs with maximum degree $\Delta$ and injective chromatic number $\lfloor \frac{3}{2} \Delta \rfloor + 1$.
\end{proof}

\begin{figure}

\centering
    \begin{minipage}{0.45\textwidth}
        \centering
            \begin{tikzpicture}[scale=1.5]
        \coordinate (u) at (0,0);
        \coordinate (w) at (3,0);
        \coordinate (v) at (1.5,1.5);
        \coordinate[below left= 0.35cm and 0.35cm of u] (ug);
        \coordinate[below right= 0.35cm and 0.35cm of w] (wg);
        
        \coordinate (a1) at (0.55,0.95);
        \coordinate (a2) at (0.35,1.15);
        \coordinate (ak) at (0,1.50);
        
        \coordinate (b1) at (2.45,0.95);
        \coordinate (b2) at (2.65,1.15);
        \coordinate (bk) at (3,1.50);
        
        \coordinate (c1) at (1.5,0.4);
        \coordinate (c2) at (1.5,0.1);
        \coordinate (ck1) at (1.5,-0.5);
        
        \draw (u) -- (v) -- (w);
        \draw (u) -- (a1) -- (v);
        \draw (u) -- (a2) -- (v);
        \draw (u) -- (ak) -- (v);
        \draw (w) -- (b1) -- (v);
        \draw (w) -- (b2) -- (v);
        \draw (w) -- (bk) -- (v);
        \draw (u) -- (c1) -- (w);
        \draw (u) -- (c2) -- (w);
        \draw (u) -- (ck1) -- (w);
        \draw (a1) -- (a2);
        \draw (b1) -- (b2);
        \draw (c1) -- (c2);
        \draw(ak) to[bend left] (bk);
        \draw(ak) to[bend right, in=-135] (ug) to[bend right, out=-45] (ck1);
        \draw(bk) to[bend left, in=135] (wg) to[bend left, out=45] (ck1);

        \draw[dotted] (a2) -- (ak);
        \draw[dotted] (b2) -- (bk);
        \draw[dotted] (c2) -- (ck1);
        
        \foreach \point in {u,v,w,a1,a2,ak,b1,b2,bk,c1,c2,ck1} {
            \fill (\point) circle (1.5pt);
        }
    
        \node[below left] at (u) {$u$};
        \node[below right] at (w) {$w$};
        \node[above] at (v) {$v$};
        \node[above] at (a1) {$a_1$};
        \node[above] at (a2) {$a_2$};
        \node[above left] at (ak) {$a_k$};
        \node[above] at (b1) {$b_1$};
        \node[above] at (b2) {$b_2$};
        \node[above right] at (bk) {$b_k$};
        \node[above right] at (c1) {$c_1$};
        \node[above right] at (c2) {$c_2$};
        \node[below right] at (ck1) {$c_{k+1}$};
    
    \end{tikzpicture}

    \end{minipage}
    \hfill
    \begin{minipage}{0.45\textwidth}
        \centering

         \begin{tikzpicture}[scale=1.5]
        \coordinate (u) at (0,0);
        \coordinate (w) at (3,0);
        \coordinate (v) at (1.5,1.5);
        
        \coordinate[below left= 0.35cm and 0.35cm of u] (ug);
        \coordinate[below right= 0.35cm and 0.35cm of w] (wg);
        
        \coordinate (a1) at (0.55,0.95);
        \coordinate (a2) at (0.35,1.15);
        \coordinate (ak) at (0,1.50);
        
        \coordinate (b1) at (2.45,0.95);
        \coordinate (b2) at (2.65,1.15);
        \coordinate (bk) at (3,1.50);
        
        \coordinate (c1) at (1.5,0.4);
        \coordinate (c2) at (1.5,0.1);
        \coordinate (ck1) at (1.5,-0.5);
        \coordinate (ck2) at (1.5,-1);

        \draw (u) -- (v) -- (w);
        \draw (u) -- (a1) -- (v);
        \draw (u) -- (a2) -- (v);
        \draw (u) -- (ak) -- (v);
        \draw (w) -- (b1) -- (v);
        \draw (w) -- (b2) -- (v);
        \draw (w) -- (bk) -- (v);
        \draw (u) -- (c1) -- (w);
        \draw (u) -- (c2) -- (w);
        \draw (u) -- (ck1) -- (w);
        \draw (u) -- (ck2) -- (w);
        \draw (a1) -- (a2);
        \draw (b1) -- (b2);
        \draw (c1) -- (c2);
        \draw (ck1) -- (ck2);
        \draw (ak) to[bend left] (bk);
        
        \draw(ak) to[bend right, in=-145] (ug) to[bend right, out=-35] (ck2);
        \draw(bk) to[bend left, in=145] (wg) to[bend left, out=35] (ck2);
        
        \draw[dotted] (a2) -- (ak);
        \draw[dotted] (b2) -- (bk);
        \draw[dotted] (c2) -- (ck1);
        
        \foreach \point in {u,v,w,a1,a2,ak,b1,b2,bk,c1,c2,ck1, ck2} {
            \fill (\point) circle (1.5pt);
        }
    
        \node[below left] at (u) {$u$};
        \node[below right] at (w) {$w$};
        \node[above] at (v) {$v$};
        \node[above] at (a1) {$a_1$};
        \node[above] at (a2) {$a_2$};
        \node[above left] at (ak) {$a_k$};
        \node[above] at (b1) {$b_1$};
        \node[above] at (b2) {$b_2$};
        \node[above right] at (bk) {$b_k$};
        \node[above right] at (c1) {$c_1$};
        \node[above right] at (c2) {$c_2$};
        \node[below right] at (ck1) {$c_{k+1}$};
         \node[below right] at (ck2) {$c_{k+2}$};
    
    \end{tikzpicture}
    \end{minipage}

    \caption{Constructions of diameter $2$ planar graphs with maximum degree $\Delta \geq 8$ (even on the left and odd on the right) and injective chromatic number equal to $\lfloor \frac{3}{2} \Delta \rfloor + 1$.}
    \label{fig:deg_+8}

\end{figure}
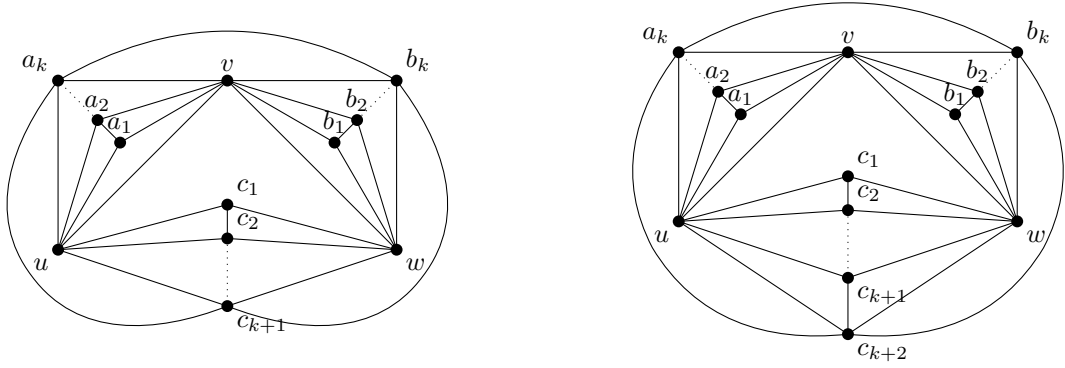



We end this section with an infinite family of planar graphs with maximum degree $3$ and injective chromatic number $5$. While not a counterexample to Conjecture~\ref{conj:chen} by Chen et al., the graphs are sharp for this conjecture as well as for Conjecture~\ref{conj:luzar}.

Let \[V_n = \{v^0_0,v^0_1\}\cup\{v^i_0, v^i_1, v^i_2, v^i_3\}^{n}_{i=1},\]
\[E_n^0 = \{v_0^0v_1^0\}\cup\{v^i_1v^i_0, v^i_1v^i_2, v^i_1v^i_3, v^i_2v^i_3\}^n_{i=1},\]
\[E_n^1 = \{v^i_0v^{i+1}_2, v^i_3v^{i+1}_0\}_{i=0}^n,\] where we take indices modulo $n+1$, $v^0_2:= v^0_1$ and $v^0_3:= v^0_1$.
We define the infinite family $\mathcal{G}':= \{(V_n, E^0_n\cup E^1_n) \mid n \geq 3 \text{ odd}\}$. It is easy to see that every $G\in \mathcal{G}'$ is planar and cubic. Carefully applying Menger's theorem shows that these graphs are $3$-connected.

\begin{theorem}
    For every graph $G \in \mathcal{G}'$, $\chi_i(G) = 5$.
\end{theorem}
\begin{proof}
    We show that $\chi_i(G) > 4$ for any $G\in \mathcal{G}'$. Let $c$ be an injective $4$-colouring of $G$.
    Then $c$ uses either $3$ or $4$ colours for $v^1_0,v^1_1,v^1_2,v^1_3$. Suppose the former, then $c(v^1_0) = c(v^1_1)$. Since $v^2_0, v^2_1, v^2_2, v^2_3$ all have a path of length $2$ to either $v^1_0$ or $v^1_1$, they must be coloured using three colours and $c(v^2_0)=c(v^2_1)$ is different from the three colours used before. Inductively and since $n$ is odd, it follows that $c(v^n_0) = c(v^n_1) = c(v^1_0) = c(v^1_1)$, which is a contradiction since $v^n_0$ and $v^1_0$ have a path of length $2$ between them. Therefore, $v^1_0,v^1_1,v^1_2,v^1_3$ all have a distinct colour, but then it is easy to see that $c(v^0_0) = c(v^1_2)$ and for the same reasoning, $c(v^0_0) = c(v^n_3)$, which is a contradiction, since $v^1_2$ and $v^n_3$ share a path of length $2$.

    Finally, it is straightforward to construct an injective $5$-colouring for $G$, which finishes the proof.
\end{proof}






We note that $\Delta = 4$, is the only maximum degree for which we did not find an infinite family which is sharp for Conjecture~\ref{conj:luzar}. The counterexample with $\Delta = 4$ and $\chi_i(H) = 9$ presented by Lu\v zar and \v Skrekovski is the only example with these properties found by our computer search. The question arises if there are other graphs with these characteristics.

\section{Higher girth}\label{sec:girth}

We note that all of our constructions so far made use of triangles.
For the injective chromatic number of planar graphs with girth greater than or equal to four, La and \v{S}torgel proposed a conjecture closely related to Conjecture~\ref{conj:luzar}. We denote the girth of a graph $G$ by $g(G)$.

\begin{conjecture}
\textup{(La, \v{S}torgel~\cite{la20232}).} Let $G$ be a planar graph with maximum degree $\Delta\geq 3$ and $g(G) \geq 4$. Then,
\[
  \chi_i(G) \leq
     \begin{cases}
       \text{$4$,}  & \text{if $\Delta = 3$;}\\
       \text{$\Delta + 2$,}  & \text{if $4 \leq \Delta \leq 5$;}\\
       \text{$\lfloor \frac{3}{2} \Delta \rfloor$,}  & \text{if $\Delta \geq 6$.}\\
     \end{cases}
\]
\label{conj:la2}
\end{conjecture}
We note that this is equivalent to saying that for a planar graph $G$ with maximum degree at least $3$ and girth at least $4$, we have $\chi_i(G) \le \lfloor \frac{3}{2}\Delta\rfloor$.

Using our algorithm from Section~\ref{sect:results}, we verified the conjecture for small orders using \texttt{nauty}'s \texttt{geng} and \texttt{planarg}~\cite{Mc14} to generate planar graphs with girth at least $4$. For order $16$ the generation took approximately 1\,000 CPU-hours, while our program for computing the injective chromatic number took approximately 4 hours. We obtain the following results.
\begin{proposition}\label{prop:g4_comp}
    Conjecture~\ref{conj:la2} holds up to order $16$.
\end{proposition}

This conjecture is based on a construction used by Lu\v zar, \v Skrekovski, and Tancer~\cite{LST09} in which they take Shannon's triangles and subdivide their edges. Given a starting triangle with maximum degree $\Delta\geq 3$, this gives a $2$-connected planar graph $G$ of girth $4$ with $\chi_i(G) = \lfloor \frac{3}{2}\Delta\rfloor$.

It is relatively easy to obtain infinitely many $2$-connected examples for each $\Delta$ which are sharp for this bound. Start with the construction given by Lu\v zar et al., connect two vertices of degree $2$ with an edge and subdivide this edge twice. Repeat the procedure by now adding an edge between the new vertices of degree $2$.  We state without proof that one can apply this construction such that these graphs still have girth $4$, maximum degree $\Delta$, are planar and have injective chromatic number equal to the base graph. See Figure~\ref{fig:example_2-conn_g4}.

\begin{figure}[!htb]
    \centering
    \begin{tikzpicture}[vertex/.style={circle,draw,minimum size=1em, fill = black, scale = 0.45}, scale = 0.8]
        \node[vertex] (0) at (0,0) {};
        \node[vertex] (1) at (4,0) {};
        \node[vertex] (2) at (2,3.5) {};

        \draw (0) -- (1) -- (2) -- (0);

        \node[vertex] (3) at ($(0)!0.5!(1)$) {};
        \node[vertex] (4) at ($(1)!0.5!(2)$) {};
        \node[vertex] (5) at ($(2)!0.5!(0)$) {};

        \node[vertex] (6) at ($(0)!{(0.5/cos(-30))}!-30:(1)$) {};
        \draw (0) -- (6) -- (1);
        \node[vertex] (7) at ($(1)!{(0.5/cos(-30))}!-30:(2)$) {};
        \draw (1) -- (7) -- (2);
        \node[vertex] (8) at ($(2)!{(0.5/cos(-30))}!-30:(0)$) {};
        \draw (2) -- (8) -- (0);

        \node[vertex] (9) at ($(0)!{(0.5/cos(-15))}!-15:(1)$) {};
        \draw (0) -- (9) -- (1);
        \node[vertex] (10) at ($(1)!{(0.5/cos(-15))}!-15:(2)$) {};
        \draw (1) -- (10) -- (2);
        \node[vertex] (11) at ($(2)!{(0.5/cos(-15))}!-15:(0)$) {};
        \draw (2) -- (11) -- (0);
    \end{tikzpicture}\qquad
    \begin{tikzpicture}[vertex/.style={circle,draw,minimum size=1em, fill = black, scale = 0.45}, scale=0.8]
        \node[vertex] (0) at (0,0) {};
        \node[vertex] (1) at (4,0) {};
        \node[vertex] (2) at (2,3.5) {};

        \draw (0) -- (1) -- (2) -- (0);

        \node[vertex] (3) at ($(0)!0.5!(1)$) {};
        \node[vertex] (4) at ($(1)!0.5!(2)$) {};
        \node[vertex] (5) at ($(2)!0.5!(0)$) {};

        \node[vertex] (6) at ($(0)!{(0.5/cos(-30))}!-30:(1)$) {};
        \draw (0) -- (6) -- (1);
        \node[vertex] (7) at ($(1)!{(0.5/cos(-30))}!-30:(2)$) {};
        \draw (1) -- (7) -- (2);
        \node[vertex] (8) at ($(2)!{(0.5/cos(-30))}!-30:(0)$) {};
        \draw (2) -- (8) -- (0);

        \node[vertex] (9) at ($(0)!{(0.5/cos(-15))}!-15:(1)$) {};
        \draw (0) -- (9) -- (1);
        \node[vertex] (10) at ($(1)!{(0.5/cos(-15))}!-15:(2)$) {};
        \draw (1) -- (10) -- (2);
        \node[vertex] (11) at ($(2)!{(0.5/cos(-15))}!-15:(0)$) {};
        \draw (2) -- (11) -- (0);

        \node[vertex] (12) at ($(3)!0.33!(4)$) {};
        \node[vertex] (13) at ($(3)!0.66!(4)$) {};
        \draw (3) -- (4);
    \end{tikzpicture} \qquad
    \begin{tikzpicture}[vertex/.style={circle,draw,minimum size=1em, fill = black, scale = 0.45}, scale=0.8]
        \node[vertex] (0) at (0,0) {};
        \node[vertex] (1) at (4,0) {};
        \node[vertex] (2) at (2,3.5) {};

        \draw (0) -- (1) -- (2) -- (0);

        \node[vertex] (3) at ($(0)!0.5!(1)$) {};
        \node[vertex] (4) at ($(1)!0.5!(2)$) {};
        \node[vertex] (5) at ($(2)!0.5!(0)$) {};

        \node[vertex] (6) at ($(0)!{(0.5/cos(-30))}!-30:(1)$) {};
        \draw (0) -- (6) -- (1);
        \node[vertex] (7) at ($(1)!{(0.5/cos(-30))}!-30:(2)$) {};
        \draw (1) -- (7) -- (2);
        \node[vertex] (8) at ($(2)!{(0.5/cos(-30))}!-30:(0)$) {};
        \draw (2) -- (8) -- (0);

        \node[vertex] (9) at ($(0)!{(0.5/cos(-15))}!-15:(1)$) {};
        \draw (0) -- (9) -- (1);
        \node[vertex] (10) at ($(1)!{(0.5/cos(-15))}!-15:(2)$) {};
        \draw (1) -- (10) -- (2);
        \node[vertex] (11) at ($(2)!{(0.5/cos(-15))}!-15:(0)$) {};
        \draw (2) -- (11) -- (0);

        \node[vertex] (12) at ($(3)!0.33!(4)$) {};
        \node[vertex] (13) at ($(3)!0.66!(4)$) {};
        \draw (3) -- (4);

        \node[vertex] (14) at ($(3)!0.33/cos(30)!30:(4)$) {};
        \node[vertex] (15) at ($(4)!0.33/cos(-30)!-30:(3)$) {};
        \draw (12) -- (14) -- (15) -- (13);
    \end{tikzpicture}
    \caption{Example for $\Delta = 6$ of a construction for obtaining infinitely many $2$-connected planar graphs of girth $4$, maximum degree $\Delta$ and injective chromatic number $\lfloor \frac{3}{2}\Delta\rfloor$.}
    \label{fig:example_2-conn_g4}
\end{figure}
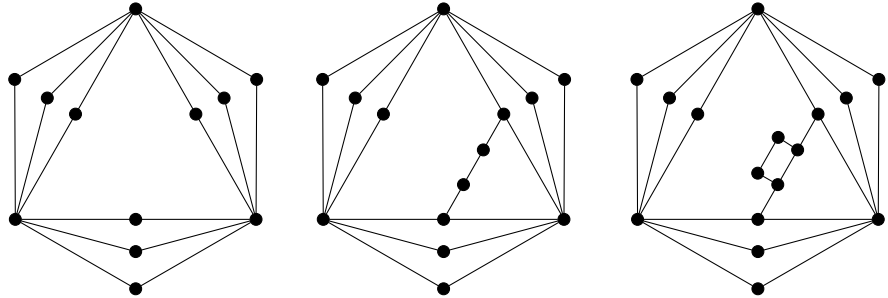


By filtering the sharp examples obtained from the computation of Proposition~\ref{prop:g4_comp} for $3$-connectivity, using \texttt{countg} from the \texttt{nauty} library~\cite{Mc14}, we obtain the following.
\begin{proposition}
    For $\Delta\geq 5$ there are no $3$-connected planar graphs of maximum degree $\Delta$ and girth at least $4$ up to order $16$ for which the injective chromatic number is sharp for Conjecture~\ref{conj:la2}.
\end{proposition}

Contrary to the previous result, we will now show that for $\Delta = 3$ and $\Delta = 4$ relatively small $3$-connected examples exist with maximum degree $\Delta$, girth at least $4$ and injective chromatic number sharp for Conjecture~\ref{conj:la2}. In particular, we even show there are infinitely many $3$-connected examples with $\Delta = 3$ and $\Delta = 4$ which are sharp for the conjecture by La and \v Storgel.

Let the $k$-prism, with $k\geq 3$, be the graph consisting of two $k$-cycles $u_0u_1\ldots u_{k-1}$ and $v_0v_1\ldots v_{k-1}$ and edges $u_iv_i$ for $0\le i \le k-1$.
These graphs are clearly $3$-connected, planar and when $k\ge4$, their girth is $4$.
\begin{theorem}
    The $k$-prism with $k \ge 3$ has injective chromatic number $4$ if and only if $k \not \equiv 0\bmod 3$.
\end{theorem}
\begin{proof}
    Let $G$ be a $k$-prism and $c$ be an injective $3$-colouring of $G$. We take indices modulo $k$. We see that the neighbours of any $u_i$ or $v_i$ need three different colours in $c$. Without loss of generality, let $u_0, u_2$ and $v_1$ have colours $1,2$ and $3$, respectively.

    Suppose $k$ is odd. Then all colours in the graph are determined, since $v_3$ must have colour $1$, which means $u_4$ must have colour $3$, etc. If $k\equiv 1\bmod 6$, then $c(u_{k-1}) = 1$. If $k\equiv 5\bmod 6$, then $c(u_{k-2}) = 1$. Both of which share a common neighbour with $u_0$, which gives us a contradiction.

    Suppose $k$ is even, then only half of the vertices' colours are determined, but if $k\equiv 2\bmod 6$, then $c(v_{k-1}) = 3$, which gives a contradiction with the colour of $v_1$ and if $k\equiv 4\bmod 6$, then $c(v_{k-1})=1$, which gives a contradiction with the colour of $u_0$.
    It is easy to find an injective $4$-colouring for the above cases.

   Finally, suppose $k\equiv 0\bmod 3$, then $c$ given by $c(u_i) = c(v_i) = j+1$ when $i = 3\alpha + j$ for $j\in \{0,1,2\}$ and $i\in \{0,\ldots, k-1\}$ is an injective $3$-colouring of $G$.
\end{proof}

\begin{figure}[!htb]
    \centering
    \begin{tikzpicture}[main_node/.style={circle,fill,inner sep=3pt, scale=0.7}, scale=0.6]

        \node[main_node, label=$u_3$] (0) at (3.1428571428571423, 0.7857142857142851) {};
        \node[main_node, label=right:$u_1$] (1) at (0.026915113871636365, 2.4026915113871636) {};
        \node[main_node, label=right:$u_2$] (2) at (0.026915113871636365, -0.829192546583851) {};
        \node[main_node] (3) at (-4.857142857142858, 0.7857142857142851) {};
        \node[main_node] (4) at (-1.74120082815735, 2.400621118012422) {};
        \node[main_node] (5) at (-1.74120082815735, -0.829192546583851) {};

        \node[main_node, label=$u$] (6) at (1.4699792960662537, 0.7857142857142851) {};
        \node[main_node] (7) at (-3.1842650103519676, 0.7857142857142851) {};
        \node[main_node] (8) at (-0.8571428571428563, 4.785714285714286) {};
        \node[main_node] (9) at (-0.8571428571428563, -3.2142857142857135) {};
        \node[main_node] (10) at (-0.8571428571428563, 0.7836438923395455) {};

         \path[draw]
        (0) edge node {} (6)
        (0) edge node {} (8)
        (0) edge node {} (9)
        (1) edge node {} (6)
        (1) edge node {} (8)
        (1) edge node {} (10)
        (2) edge node {} (6)
        (2) edge node {} (9)
        (2) edge node {} (10)
        (3) edge node {} (7)
        (3) edge node {} (8)
        (3) edge node {} (9)
        (4) edge node {} (7)
        (4) edge node {} (8)
        (4) edge node {} (10)
        (5) edge node {} (7)
        (5) edge node {} (9)
        (5) edge node {} (10)
        ;

    \end{tikzpicture}
\caption{Graph with maximum degree $4$, girth $4$ and injective chromatic number $6$ used as a base graph for the infinite family of Theorem~\ref{thm:g4family}.}
\label{fig:g4_base_graph}
\end{figure}
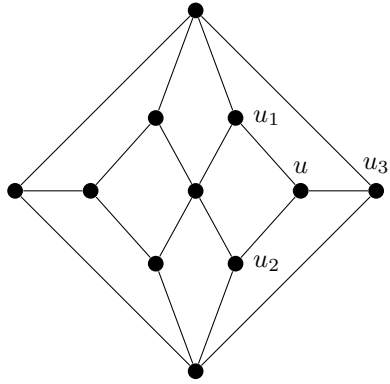

Let $G$ be a graph with a vertex $u$ of degree $3$ and neighbours $u_1, u_2$ and $u_3$. Let $V = V(G)\cup \{v_1, v_2, v_3\}$ and $E = \{v_1v_3,v_2v_3, v_1u_1,v_1u_3,v_2u_2,v_2u_3,v_3u\}\cup E(G)\setminus\{uu_3\}$. We obtain a new graph $G' = (V,E)$. Note that $v_3$ is now a vertex of degree $3$ in $G'$. We can relabel the graph so that $u, v_1, v_2, v_3$ become $u_3, u_2, u_1, u$, respectively. In this way we can iterate the procedure to obtain an infinite family $\mathcal{K}(G; u; u_1, u_2, u_3)$.
It is easy to see that if $G$ is planar then all members of the infinite family are planar, that if $G$ has girth $4$ then all members of the family have girth $4$, and that if $\Delta(G) = 4$ and $u_1, u_2$ and $u_3$ are all of degree $3$ then every member of the family has maximum degree $4$. If $G$ is $3$-connected, one can show using Menger's theorem that all members of the family are also $3$-connected.
\begin{theorem}\label{thm:g4family}
    There is an infinite family of $3$-connected planar graphs with maximum degree $4$, girth $4$ and injective chromatic number $6$.
\end{theorem}
\begin{proof}
    Let $G$ be the graph\footnote{{This graph can be inspected on the House of Graphs~\cite{CDG23} at \url{https://houseofgraphs.org/graphs/1158}}.} of Figure~\ref{fig:g4_base_graph} and let $u, u_1, u_2, u_3$ be the vertices labelled in the figure. Then $\mathcal{K} := \mathcal{K}(G;u; u_1, u_2, u_3)$ is an infinite family of $3$-connected planar graphs of maximum degree $4$.

    It is easy to verify by computer that $\chi_i(G - uu_3) > 5$, since this is a subgraph of all $H\in \mathcal{K}$ it follows that $\chi_i(H)\geq 6$.
    It is straightforward to check that $G$ has an injective $6$-colouring. Let $H$ be any member of $\mathcal{K}$ with an injective $6$-colouring $c$ and let $H'$ be its successor by applying the procedure to vertices $u^H, u^H_1, u^H_2, u^H_3$ of $H$. Let $v_1, v_2, v_3$ be the new vertices of $H'$ as described by the procedure. We construct an injective $6$-colouring $c'$ on $H'$. Let $c'(v) = c(v)$ for all $v\in V(H)$. We see that $c'(v_1)$ cannot have the colour of $u$, of $v_2$ and of the neighbours of $u_1$ and $u_3$. Since $u_1$ and $u_3$ are of degree $3$ and share a common neighbour which is not $v_1, v_2$ or $u$, this restricts at most $5$ colours. Similarly, we have a free colour for $v_2$ and three free colours for $v_3$. Hence, we can find an injective $6$-colouring for $H'$. By induction it follows that every member of $\mathcal{K}$ has injective chromatic number $6$.
\end{proof}

A related question for girth at least five was asked by Lu\v{z}ar, \v{S}krekovski, and Tancer~\cite{LST09}.

\begin{problem}
\textup{(Lu\v{z}ar, \v{S}krekovski, and Tancer).} Does an integer $M$ exist such that every planar graph $G$ with maximum degree $\Delta \geq M$ and girth at least $5$ is injectively ($\Delta + 1$)-colourable?
\label{prob:luzar}
\end{problem}

Using our algorithm for computing the injective chromatic number and $\texttt{geng}$ and $\texttt{planarg}$~\cite{Mc14}
for generating planar graphs of girth at least $5$, we find the following.

\begin{proposition}\label{prop:g5_comp}
    Up to order $21$ every planar graph with girth at least $5$ is injectively $(\Delta + 1)$-colourable.
\end{proposition}
The computation for order $21$ took approximately 16\,000 hours for the generation and 48 hours for determining the injective chromatic number of the graphs.

In~\cite{BI11} Borodin and Ivanova gave in their Figure~1, for every $\Delta \ge 3$, an example of a $2$-connected planar graph with girth $6$ and maximum degree $\Delta$ which is injectively $(\Delta + 1)$-colourable. Smoothing one of the degree $2$ vertices in their example gives a family with the same properties but of girth $5$. A similar idea to the one used
in the girth at least $4$ case gives infinite families of planar $2$-connected injectively $(\Delta + 1)$-colourable graphs with girth $6$ for every $\Delta \geq 3$.

Also in this case it is difficult to find $3$-connected examples.
By filtering the sharp examples obtained from the computation of Proposition~\ref{prop:g5_comp} for $3$-connectivity, using \texttt{countg} from the \texttt{nauty} library~\cite{Mc14}, we obtain the following.
\begin{proposition}
    For $\Delta \ge 4$ there are no $3$-connected planar graphs of maximum degree $\Delta$ and girth at least $5$ for which the injective chromatic number equals $\Delta + 1$ up to order $21$. When $\Delta = 3$, the smallest $3$-connected planar graph of girth at least $5$ whose injective chromatic number is $4$ is the dodecahedron graph on $20$ vertices.
\end{proposition}

Let $r > 2$. The generalised dodecahedron $D_r$ has vertex set $$\{u_0,\ldots, u_{r-1}, v_0,\ldots, v_{r-1}, u'_0,\ldots, u'_{r-1}, v'_0, \ldots, v'_{r-1}\}$$ and edge set $$\{u_iu_{i+1},v_iv_{i+1}, u_iu'_i, v_iv'_i, v_i'u_i', u'_iv'_{i+1}\}_{i=0}^{r-1}$$ where indices are taken modulo $r$. A visualisation of this class of graphs was given in Figure~3 in~\cite{CJM21} (using a different labelling).

It is easy to see that the generalised dodecahedra are $3$-connected cubic planar graphs with two $r$-gonal faces and $2r$ pentagonal faces and hence are of girth $5$ when $r\geq 5$.

Denote the vertices of one of the $r$-gonal faces by $u_0, u_1, \ldots, u_{r-1}$ and the vertices of the other by $v_0, v_1, \ldots, v_{r-1}$. For each $u_i, v_i$ denote the remaining vertex by $u_i'$, $v_i'$, respectively. We do this in such a way that $u_i'$ is adjacent to $v_i'$ and $v_{i+1}'$, taking indices modulo $r$.

\begin{theorem}
    The generalised dodecahedra $D_r$ with $r\geq 5$ have $\chi_i(D_r) = 4$ if and only if $r\not\equiv 0 \bmod 3$.
\end{theorem}
\begin{proof}
    Let $D'$ be the graph obtained from $D_r$ by removing one of the $r$-gonal cycles, say $u_0u_1\ldots u_{r-1}$. Suppose $D'$ is injectively $3$-colourable. We take indices modulo $r$. If in one of the pentagons $v_iv_i'u_i'v_{i+1}'v_{i+1}$, the same colour is given to $u_i'$ and $v_{i+1}'$ and to $v_i$ and $v_{i+1}$, then $u_{i+1}'$ and $v_{i+2}$ have the same colour which is a contradiction. Similarly, if the same colour is given to $u_i'$ and $v_{i+1}'$ but not to $v_i$ and $v_{i+1}$, then the same colour is given to $u_{i+1}'$ and $v_{i+2}'$ and to $v_{i+1}$ and $v_{i+2}$, again leading to a contradiction as before. By symmetry, $u_i'$ and $v_i'$ need to have different colours.

    Without loss of generality, $D'$ must be coloured such that $v_i$ and $v_i'$ get colour $j\in\{1,2,3\}$ when $i=3\alpha + j$. If $r\equiv 1\bmod 3$, then $v_0$ and $v_{r-1}'$ have the same colour. If $r\equiv 2\bmod 3$, then $v_0$ and $v_{r-2}$ have the same colour, leading to a contradiction. It is easy to find an injective $4$-colouring in these cases.

    Finally, suppose $r\equiv 0\bmod 3$, then the previous idea leads to an injective $3$-colouring if we give $u_i'$ the colour different from that of $v_i'$ and $v_{i+1}'$.
\end{proof}


\section*{Acknowledgements}

We thank Borut Lu\v zar for interesting discussions on this topic.

Jan Goedgebeur and Jarne Renders are supported by an FWO grant with grant number G0AGX24N and by Internal Funds of KU Leuven.

The computational resources and services used in this work were provided by the VSC (Flemish Supercomputer Center), funded by the Research Foundation - Flanders (FWO) and the Flemish Government – department EWI.

\newpage


\bibliographystyle{plain}
\bibliography{references}

@techreport{wegner1977graphs,
  title={Graphs with given diameter and a coloring problem},
  author={Wegner, G.},
  institution = {University of Dortmund},
  year={1977}
}

@article{luvzar2015counterexamples,
  title={Counterexamples to a conjecture on injective colorings},
  author={Lu{\v{z}}ar, B. and {\v{S}}krekovski, R.},
  journal={Ars Mathematica Contemporanea},
  volume={8},
  number={2},
  pages={291–295},
  year={2015}
}

@article{chen2012some,
  title={Some results on the injective chromatic number of graphs},
  author={Chen, M. and Hahn, G. and Raspaud, A. and Wang, W.},
  journal={Journal of Combinatorial Optimization},
  volume={24},
  number={3},
  pages={299--318},
  year={2012},
  publisher={Springer}
}

@article{HKSS02,
  title = {On the injective chromatic number of graphs},
  volume = {256},
  doi = {10.1016/S0012-365X(01)00466-6},
  number = {1-2},
  journal = {Discrete Mathematics},
  author = {Hahn, G. and Kratochv\'il, J. and Širáň, J. and Sotteau, D.},
  year = {2002},
  pages = {179--192},
}

@article{la20232,
  title={2-distance, injective, and exact square list-coloring of planar graphs with maximum degree 4},
  author={La, H. and {\v{S}}torgel, K. },
  journal={Discrete Mathematics},
  volume={346},
  number={8},
  pages={113405},
  year={2023},
  publisher={Elsevier}
}

@article{brinkmann2007program,
  title={Fast generation of planar graphs},
  author={Brinkmann, G. and McKay, B. D.},
  journal={MATCH-Communications In Mathematical And In Computer Chemistry},
  volume={58},
  number={2},
  pages={323--357},
  year={2007}
}

@article{DL13,
  title = {Injective coloring of planar graphs with girth 6},
  volume = {313},
  doi = {10.1016/j.disc.2013.02.014},
  number = {12},
  journal = {Discrete Mathematics},
  author = {Dong, W. and Lin, W. },
  year = {2013},
  pages = {1302--1311},
}

@misc{DGR24Program,
	author = {M. Daneels and J. Goedgebeur and J. Renders},
	month = {1},
	title = {{injChromNumber}},
	url = {https://github.com/mtsdaneels/injChromNumber},
	note = {https://github.com/mtsdaneels/injChromNumber},
	version = {1},
	year = {2024}
	}

@article{Mc14,
	title = {Practical graph isomorphism, {II}},
	volume = {60},
	doi = {10.1016/j.jsc.2013.09.003},
	journal = {Journal of Symbolic Computation},
	author = {McKay, B. D. and Piperno, A.},
	year = {2014},
}

@article{BI11,
	title = {List injective colorings of planar graphs},
	volume = {311},
	doi = {10.1016/j.disc.2010.10.008},
	number = {2},
	journal = {Discrete Mathematics},
	author = {Borodin, O. V. and Ivanova, A. O.},
	year = {2011},
	pages = {154--165},
}

@article{JX20,
	title = {Connectivity of {Minimum} {Non}-5-injectively {Colorable} {Planar} {Cubic} {Graphs}},
	volume = {8},
	doi = {10.1007/s40305-018-0214-6},
	number = {1},
	journal = {Journal of the Operations Research Society of China},
	author = {Jin, J. and Xu, B.},
	year = {2020},
	pages = {105--116},
}

@article{MS05,
	title = {A bound on the chromatic number of the square of a planar graph},
	volume = {94},
	doi = {10.1016/j.jctb.2004.12.005},
	number = {2},
	journal = {Journal of Combinatorial Theory, Series B},
	author = {Molloy, M. and Salavatipour, M. R.},
	year = {2005},
	pages = {189--213},
}

@article{DL14,
	title = {Injective coloring of plane graphs with girth 5},
	volume = {315-316},
	doi = {10.1016/j.disc.2013.10.017},
	journal = {Discrete Mathematics},
	author = {Dong, W. and Lin, W.},
	year = {2014},
	pages = {120--127},
}

@article{BY22,
	title = {Injective coloring of planar graphs with girth 5},
	volume = {17},
	doi = {10.1007/s11464-022-1018-x},
	number = {3},
	journal = {Frontiers of Mathematics in China},
	author = {Bu, Y. and Ye, P.},
	year = {2022},
	pages = {473--484},
}

@article{LST09,
	series = {Combinatorics 2006, {A} {Meeting} in {Celebration} of {Pavol} {Hell}’s 60th {Birthday} ({May} 1–5, 2006)},
	title = {Injective colorings of planar graphs with few colors},
	volume = {309},
	doi = {10.1016/j.disc.2008.04.005},
	number = {18},
	journal = {Discrete Mathematics},
	author = {Lužar, B. and Škrekovski, R. and Tancer, M.},
	year = {2009},
	pages = {5636--5649},
}

@article{CJM21,
	title = {Facial {Visibility} in {Edge} {Colored} {Plane} {Graphs}},
	volume = {38},
	doi = {10.1007/s00373-021-02411-9},
	number = {1},
	journal = {Graphs and Combinatorics},
	author = {Czap, J. and Jendrol’, S. and Madaras, T.},
	year = {2021},
	pages = {4},
}

@article{CDG23,
	title = {House of {Graphs} 2.0: {A} database of interesting graphs and more},
	volume = {325},
	shorttitle = {House of {Graphs} 2.0},
	doi = {10.1016/j.dam.2022.10.013},
	language = {en},
	journal = {Discrete Applied Mathematics},
	author = {Coolsaet, K. and D’hondt, S. and Goedgebeur, J.},
	year = {2023},
	pages = {97--107},
}

@article{Wh33,
	title = {2-{Isomorphic} {Graphs}},
	volume = {55},
	doi = {10.2307/2371127},
	number = {1},
	journal = {American Journal of Mathematics},
	author = {Whitney, H.},
	year = {1933},
	pages = {245--254},
}

\newpage
\appendix

\section{Correctness tests}\label{app:correctness}
While it is relatively easy to prove the correctness of our algorithm to compute the chromatic number of a graph presented in Section~\ref{sect:results},
we also performed various tests for verifying the correctness of its implementation. Our implementation of the algorithm is open source software and can be found on
GitHub~\cite{DGR24Program} where it can be verified and used by other researchers.

Let $G$ be a graph and let $G^{(2)}$ be the \emph{neighbouring graph} with vertex set $V(G)$ and with edge set $\{uv; N_G(u)\cap N_G(v) \neq \emptyset\}$. It follows that $\chi_i(G)$ corresponds to the chromatic number of $G^{(2)}$. We wrote a program that, given $G$, computes $G^{(2)}$. Together with a program for computing the chromatic number -- we use one which is part of the \texttt{nauty} library~\cite{Mc14} -- we can reproduce the experiments performed in the paper up to certain orders. However, this approach is not as fast as the one described in Section~\ref{sect:results}.
In particular, we checked the counts of graphs attaining a specific injective chromatic number for connected planar graphs of minimum degree $2$ up to order $11$, of girth at least $4$ up to order $14$ and of girth at least $5$ up to order $17$, of $3$-connected planar graphs up to order $12$, of $4$-connected planar graphs up to order $15$ and of $3$-connected cubic planar graphs up to order $28$. All results were in complete agreement with those of the algorithm from Section~\ref{sect:results}.

Our backtracking algorithm
can also be adapted to compute the chromatic number of graphs and we verified for a selection of graphs that this indeed yields the correct chromatic number.


\section{The smallest most symmetric graphs of Table~\ref{tab:sharp_for_LS}}\label{app:hog}
The smallest most symmetric graphs of Table~\ref{tab:sharp_for_LS} for each maximum degree $3\leq \Delta\leq 8$  can among others be inspected on the House of Graphs~\cite{CDG23} by searching for the keywords ``injective chromatic number''. Here we list their specific URLs.

\begin{itemize}
    \item $\Delta = 3$: \url{https://houseofgraphs.org/graphs/52993}
    \item $\Delta = 4$: \url{https://houseofgraphs.org/graphs/33503}
    \item $\Delta = 5$: \url{https://houseofgraphs.org/graphs/52994}
    \item $\Delta = 6$: \url{https://houseofgraphs.org/graphs/52995}, \\\url{https://houseofgraphs.org/graphs/52996}
    \item $\Delta = 7$: \url{https://houseofgraphs.org/graphs/52999}
    \item $\Delta = 8$: \url{https://houseofgraphs.org/graphs/53043}
\end{itemize}

\section{A direct proof of injective chromatic number of two graphs used in the proof of Theorem~\ref{theor:pentagonFamily}}\label{app:two_graphs_inj_at_least_7}

Here we give a direct (but involved) argument of a fact used in the proof of Theorem~\ref{theor:pentagonFamily}. Let $G$ be the graph of Figure~\ref{graph:G1-triangle}. We will use the labels in the figure for the following proof. In the proof of Theorem~\ref{theor:pentagonFamily}, $G$ is isomorphic to $G_1 - v_0^1v_1^1 - v_0^1v_2^1 - v_1^1v_2^1$ and $G':=G\cup\{v_1d_5, d_5v_5, v_1v_5\}-w_1-w_2-w_3$ is isomorphic to $G_0$.

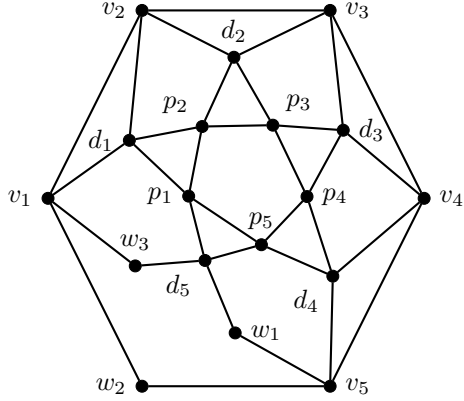
\begin{figure}[!htb]
\centering
\begin{tikzpicture}[main_node/.style={circle,draw,minimum size=1em, fill = black, scale = 0.45}, scale = 0.5]

\node[main_node, label=above right:{$p_3$}] (0) at (1.1203808785672766, 1.7320230105464998) {};
\node[main_node, label=right:{$d_3$}] (1) at (2.9921111633957986, 1.597794822627038) {};
\node[main_node, label={$d_2$}] (2) at (0.08569767763722425, 3.5396292745286027) {};
\node[main_node, label=above left:{$p_2$}] (3) at (-0.7590998422232675, 1.6962288271013106) {};
\node[main_node, label=right:{$p_4$}] (4) at (2.0291194951144584, -0.16730584851390162) {};
\node[main_node, label=right:{$v_4$}] (5) at (5.142857142857142, -0.21428571428571264) {};
\node[main_node, label=right:{$v_3$}] (6) at (2.6433415451047706, 4.785714285714286) {};
\node[main_node, label=left:{$v_2$}] (7) at (-2.3576272593904823, 4.785714285714286) {};
\node[main_node, label=left:{$d_1$}] (8) at (-2.694771223738477, 1.3293384467881109) {};
\node[main_node, label=left:{$p_1$}] (9) at (-1.1175575054668236, -0.15612016618728042) {};
\node[main_node, label={$p_5$}] (10) at (0.8161762670578803, -1.449185043144774) {};
\node[main_node, label=below left:{$d_4$}] (11) at (2.71503307775348, -2.285874081176093) {};
\node[main_node, label=right:{$v_5$}] (12) at (2.6433415451047706, -5.2142857142857135) {};
\node[main_node, label=left:{$v_1$}] (13) at (-4.857142857142858, -0.21428571428571264) {};
\node[main_node, label=below left:{$d_5$}] (14) at (-0.683533091593544, -1.867529562160434) {};
\node[main_node, label=right:{$w_1$}] (15) at (0.11863703047582241, -3.7959411952700535) {};
\node[main_node, label=left:{$w_2$}] (16) at (-2.3576272593904823, -5.2142857142857135) {};
\node[main_node, label={$w_3$}] (17) at (-2.5378248955075144, -2.0129434324065194) {};

 \path[draw, thick]
(0) edge node {} (1)
(0) edge node {} (2)
(0) edge node {} (3)
(0) edge node {} (4)
(1) edge node {} (4)
(1) edge node {} (5)
(1) edge node {} (6)
(2) edge node {} (3)
(2) edge node {} (6)
(2) edge node {} (7)
(3) edge node {} (8)
(3) edge node {} (9)
(4) edge node {} (10)
(4) edge node {} (11)
(5) edge node {} (6)
(5) edge node {} (11)
(5) edge node {} (12)
(6) edge node {} (7)
(7) edge node {} (8)
(7) edge node {} (13)
(8) edge node {} (9)
(8) edge node {} (13)
(9) edge node {} (10)
(9) edge node {} (14)
(10) edge node {} (11)
(10) edge node {} (14)
(11) edge node {} (12)
(12) edge node {} (15)
(12) edge node {} (16)
(13) edge node {} (16)
(13) edge node {} (17)
(14) edge node {} (15)
(14) edge node {} (17)
;
\end{tikzpicture}
\caption{The graph $G$ isomorphic to $G_1 - v_0^1v_1^1 - v_0^1v_2^1 - v_1^1v_2^1$.}
\label{graph:G1-triangle}
\end{figure}

\begin{lemma}
    The injective chromatic number of $G$ depicted in Figure~\ref{graph:G1-triangle} is at least $8$.
\end{lemma}
\begin{proof}
    Suppose $G$ has an injective $7$-colouring $c$. We use the same labels as in Figure~\ref{graph:G1-triangle}.

    Note that $v_1, v_2, d_1, d_2, p_1, p_2$ have pairwise distinct colours in $c$. Since $c$ uses at most $7$ colours we have three options, either $c(v_3)=c(p_1)$ and $p_3$ gets a seventh colour, $c(p_3)=c(v_1)$ and $v_3$ gets a seventh colour, or both $c(v_3)=c(p_1)$ and $c(p_3)=c(v_1)$.

    Suppose $c(v_3)=c(p_1)$ and $c(p_3)=c(v_1)$. Due to the same arguments, but for $v_3$, $v_4$, $v_5$, $d_3$, $d_4$, $p_3$, $p_4$, $p_5$, we see that $c(v_3) = c(p_5)$ and/or $c(p_3)=c(v_5)$, which leads to a contradiction as $p_5$ and $p_1$ share a common neighbour and $v_5$ and $v_1$ share a common neighbour.

    Hence, we are in one of the previous cases. Assume $c(v_1)=c(p_3)$. If $c(p_3) = c(v_5)$, we get a contradiction, so we have that $c(v_3)=c(p_5)$. It must hold in $c$ that $p_1, p_2, p_3, p_4$ and $p_5$ get pairwise distinct colours. Since $d_5$ shares a common neighbour with $p_1, p_2, p_4, p_5$ and $v_1$, $c$ uses a sixth colour for $d_5$. We see that one of $d_1, v_2, d_2$ needs to use a seventh colour as $d_1, v_2, d_2, p_1, p_2, p_3$ are pairwise injective neighbours, so need six different colours and they cannot use $c(v_3)=c(p_5)$.
    If $d_1$ has the seventh colour, then $c(d_2)=c(d_5)$ and $c(v_2)=c(p_4)$. However, $d_3,d_4,v_4,v_5$ need four different colours, but only have three possibilities left.
    If $v_2$ or $d_2$ has the seventh colour, then a similar reasoning shows that $d_3,d_4,v_4,v_5$ can only use three colours in $c$.

    Finally, assume that $c(p_1) = c(v_3)$. Then $c(p_3) = c(v_5)$. By symmetry, we also get a contradiction in this case. Hence, no injective $7$-colouring can exist for $G$.
\end{proof}
\begin{lemma}
    The injective chromatic number of $G'$ is at least $8$.
\end{lemma}
\begin{proof}
    One can see that the exact same proof method also works for $G'$ as $d_5, v_1$ and $v_5$ still pairwise share a common neighbour in $G'$.
\end{proof}

\clearpage

\end{document}